\newif\ifshowflag
\showflagtrue   % show messages
\newif\ifshowqstn
\showqstntrue   % show questions
\showqstnfalse   % hide questions
\newif\ifshowinfo
\showinfotrue   % show extra information
\showinfofalse   % hide extra information

\newcommand{\flag}[1]{\ifshowflag
  {\noindent $\clubsuit\clubsuit\clubsuit$\; {\sffamily #1}}\; $\clubsuit\clubsuit\clubsuit$\fi}
\newcommand{\showinfo}[1]{\ifshowqstn{\noindent*** Some Details ***  \\ {#1} \\  \noindent*** That's All! ***}\fi}
\documentclass[12pt]{amsart} %
\usepackage{amssymb,amscd}
\usepackage{graphicx,xcolor}
\usepackage{overpic}

\numberwithin{equation}{section}        % to get #'s with section # (good numbering!)

\newcommand{\Va}{V\"ais\"al\"a}     % why not!

\catcode`\@=11       % changed from 12 (other) to 11 (letter)

\def\rf#1{\@rf{#1}#1:;;}
\def\rfs#1{\@rfs{#1}#1:;;}
\def\rfm#1{\@rfF#1<>;;}

\def\@C{C}\def\@CC{CC}\def\@E{E}\def\@F{F}\def\@L{L}\def\@P{P}\def\@PP{PP}\def\@Q{Q}
\def\@R{R}\def\@S{S}\def\@T{T}\def\@TT{TT}\def\@X{X}\def\@XX{XX}\def\@Ex{Ex}\def\@s{s}\def\@f{f}

\def\@rf#1#2:#3;;{\def\@b{#2}% the engine behind \rf
  \ifx\@b\@C Corollary~\ref{#1}\else%
  \ifx\@b\@CC Corollary~\ref{#1}\else%
  \ifx\@b\@E (\ref{#1})\else%  equation
  \ifx\@b\@Ex Exercise~\ref{#1}\else%
  \ifx\@b\@F Fact~\ref{#1}\else%
  \ifx\@b\@L Lemma~\ref{#1}\else%
  \ifx\@b\@P Proposition~\ref{#1}\else%
  \ifx\@b\@PP Proposition~\ref{#1}\else%
  \ifx\@b\@Q Question~\ref{#1}\else%
  \ifx\@b\@R Remark~\ref{#1}\else%
  \ifx\@b\@S Section~\ref{#1}\else%
  \ifx\@b\@T Theorem~\ref{#1}\else%
  \ifx\@b\@TT Theorem~\ref{#1}\else%
  \ifx\@b\@X Example~\ref{#1}\else%
  \ifx\@b\@XX Example~\ref{#1}\else%
  \ifx\@b\@s \S\ref{#1}\else% subsection
  \ifx\@b\@f Figure~\ref{#1}\else%
  \ref{#1}\fi\fi\fi\fi\fi\fi\fi\fi\fi\fi\fi\fi\fi\fi\fi\fi\fi}

\def\@rfs#1#2:#3;;{\def\@b{#2}% the engine behind \rfs
  \ifx\@b\@C Corollaries~\ref{#1}\else%
  \ifx\@b\@CC Corollaries~\ref{#1}\else%
  \ifx\@b\@Ex Exercises~\ref{#1}\else%
  \ifx\@b\@F Facts~\ref{#1}\else%
  \ifx\@b\@L Lemmas~\ref{#1}\else%
  \ifx\@b\@P Propositions~\ref{#1}\else%
  \ifx\@b\@PP Propositions~\ref{#1}\else%
  \ifx\@b\@Q Questions~\ref{#1}\else%
  \ifx\@b\@R Remarks~\ref{#1}\else%
  \ifx\@b\@S Sections~\ref{#1}\else%
  \ifx\@b\@T Theorems~\ref{#1}\else%
  \ifx\@b\@TT Theorems~\ref{#1}\else%
  \ifx\@b\@X Examples~\ref{#1}\else%
  \ifx\@b\@XX Example~\ref{#1}\else%
  \ifx\@b\@s \S\ref{#1}\else% subsection
  \ifx\@b\@f Figures~\ref{#1}\else%
  \ref{#1}\fi\fi\fi\fi\fi\fi\fi\fi\fi\fi\fi\fi\fi\fi\fi\fi}

\def\@rfF<#1>#2;;{\def\@c{#2}% multiple refs: first one
  \@rfs{#1}#1:;;\ifx\@c\empty\else\@rfL:#2;;\fi}

\def\@rfL:#1<#2>#3;;{\def\@b{#2}\def\@c{#3}% multiple refs: later ones
  #1\ifx\@b\empty\else\ref{#2}\ifx\@c\empty\else\@rfL:#3;;\fi\fi}

\catcode`\@=12       % changed to 12 (other) from 11 (letter)

\definecolor{darkblue}{rgb}{0,0,0.6}
\definecolor{darkgreen}{rgb}{0,0.4,0}
\definecolor{darkred}{rgb}{0.6,0,0}
\definecolor{lightblue}{rgb}{0.8,0.8,1}
\definecolor{lightgreen}{rgb}{0.25,1,0.25}
\definecolor{lightred}{rgb}{1,0.5,0.5}
\definecolor{lightpurple}{rgb}{1,0.4,0.6}
\definecolor{darkpurple}{rgb}{0.5,0,0.5}

\newcommand{\ds}{\displaystyle} \newcommand{\half}{\frac{1}{2}}       
      
\DeclareMathOperator{\id}{\mathsf{id}}	      % identity map :-)
\DeclareMathOperator{\ed}{\lvert\cdot\rvert}  % Euclidean distance
\newcommand{\sm}{\setminus}		              % :-)
\providecommand{\abs}[1]{\lvert#1\rvert}	  % absolute value
\def\vint_#1{\mathchoice
          {\mathop{\vrule width 6pt height 3 pt depth -2.5pt
                  \kern -8pt \intop}\nolimits_{\kern -4pt#1}}%
          {\mathop{\vrule width 5pt height 3 pt depth -2.6pt
                  \kern -6pt \intop}\nolimits_{#1}}%
          {\mathop{\vrule width 5pt height 3 pt depth -2.6pt
                  \kern -6pt \intop}\nolimits_{#1}}%
          {\mathop{\vrule width 5pt height 3 pt depth -2.6pt
                   \kern -6pt \intop}\nolimits_{#1}}}
\newcommand{\ifff}{if and only if }  \newcommand{\wrt}{with respect to }  
\newcommand{\tfae}{the following are equivalent}

\newcommand{\bt}{bounded turning}   
\newcommand{\bl}{bi-Lipschitz}       \newcommand{\qc}{quasiconformal}
\newcommand{\qsy}{quasisymmetry}    \newcommand{\qsc}{quasisymmetric}
    
\newcommand{\homeo}{homeomorphism}  \newcommand{\homic}{homeomorphic}

\newcommand{\alf}{\alpha}       \newcommand{\del}{\delta}   \newcommand{\Del}{\Delta}
\newcommand{\veps}{\varepsilon} \newcommand{\vphi}{\varphi} \newcommand{\gam}{\gamma}
\newcommand{\Gam}{\Gamma}          \newcommand{\lam}{\lambda}
          
\newcommand{\sig}{\sigma}          \newcommand{\tha}{\theta}

\newcommand{\mcA}{{\mathcal A}} \newcommand{\mcB}{{\mathcal B}}  
 \newcommand{\mcF}{{\mathcal F}}   \newcommand{\mcI}{{\mathcal I}} \newcommand{\mcJ}{{\mathcal J}}  
   
 \newcommand{\mcR}{{\mathcal R}} \newcommand{\mcS}{{\mathcal S}}

\newcommand{\lp}{\left(}    % left paranthesis
\newcommand{\rp}{\right)}   % right paranthesis
\newcommand{\dimA}{\dim_{\mathcal A}}   % Assouad dimension
\DeclareMathOperator{\card}{card}
\DeclareMathOperator{\diam}{diam}
\DeclareMathOperator{\dist}{dist}

\newcommand{\A}{\mathsf{A}}     % can change to "regular" or \mathbf or \mathbb or \mathcal if want
\newcommand{\mathfont}{\mathsf} % math serif, or  %
\newcommand{\mfB}{{\mathfont B}}      % unit ball
\newcommand{\mfN}{{\mathfont N}}      % natural numbers
\newcommand{\mfR}{{\mathfont R}}      % real number field
\newcommand{\mfS}{{\mathfont S}}      % unit sphere
\newcommand{\bd}{\partial}      % boundary notation
\theoremstyle{definition}

\theoremstyle{remark}
\theoremstyle{plain}
\newtheorem*{thm*}{Theorem}         % for Theorems, Lemmas,
\newtheorem*{lma*}{Lemma}           % etc with NO numbering
\newtheorem*{cor*}{Corollary}
\newtheorem*{conj*}{Conjecture}
\newtheorem*{prop*}{Proposition}

\theoremstyle{remark}
\newtheorem*{claim*}{Claim}
\newtheorem*{xx*}{Example}
\newtheorem*{xxs*}{Examples}
\newtheorem*{fact*}{Fact}
\newtheorem*{qstn*}{Question}
\newtheorem*{rmk*}{Remark}
\newtheorem*{rmks*}{Remarks}

\swapnumbers                % to put numbers in front of....
\theoremstyle{plain}
\newtheorem{thm}[equation]{Theorem}
\newtheorem{lma}[equation]{Lemma}
\newtheorem{cor}[equation]{Corollary}
\newtheorem{prop}[equation]{Proposition}

\theoremstyle{remark}

\newtheorem{qstn}[equation]{Question}

\newtheorem{rmk}[equation]{Remark}

  {\par\smallskip\noindent{\em #1}}{\par\smallskip}

\newenvironment{noname}[1]%------display but with no name nor number
  {\par\smallskip\noindent%
   \leftskip=\nnlen\rightskip=\nnlen\addtolength{\leftmargini}{\nnlen}%
   \em #1}%
  {\par\smallskip\addtolength{\leftmargini}{-\nnlen}}
\newlength{\nnlen}

\newenvironment{pf}[1]%----------put out a reference number for a Proof of ...
  {\par\smallskip\noindent\refstepcounter{equation}\theequation.{ \em #1.}}%
  {\qed\smallskip}
\newenvironment{pf*}[1]{\subsubsection*{#1}}{\qed\smallskip}	% no number
\newcounter{aenumctr} % counter for new list environment
  {\begin{list}%
    {\rm(\alph{aenumctr})}% description of labels--force
    {\usecounter{aenumctr}}}  % both roman (alpha)
  {\end{list}}

\renewcommand{\bl}{bi-Lipschitz}          % Daniel's choice
\DeclareMathOperator{\dia}{\mathsf{dd}}   % use for inner diam distance---change here if needed
\DeclareMathOperator{\D}{\delta}          % use to approx diam_d(A)---change here if you want
\newcommand{\mfA}{{\mathsf A}}            %
\begin{document} %>>--BEGIN--<<%%
%%===============%=============%%
\title[Quasi and BT Circles Modulo bi-Lipschitz Maps]{Quasicircles and Bounded Turning Circles Modulo bi-Lipschitz Maps}
\date{\today}

%--------------------%
% author information %
%--------------------%

\author{David A Herron} % first author
\address{Department of Mathematics, University of Cincinnati, OH  45221}
\email{David.Herron@math.UC.edu}

\author{Daniel Meyer}% second author
\address{Department of Mathematics and Statistics, University of Helsinki, P.O. Box 68 (Gustaf H\"allstr\"omin katu 2b) FI-00014,  Helsinki, Finland}
\email{DMeyermail@gmail.com}

\thanks{The first author was partially supported by the Charles Phelps Taft Research Center.  The second author was supported by the Academy of Finland, projects SA-134757 and SA-118634.}

%----------%
% AMS info %
%----------%

%----------%
% Abstract %
%----------%
\begin{abstract}
We construct a catalog, of snowflake type metric circles, that describes all metric quasicircles up to \bl\ equivalence.  This is a metric space analog of a result due to Rohde.  Our construction also works for all bounded turning metric circles; these need not be doubling.  As a byproduct, we show that a metric quasicircle with Assouad dimension strictly less than two is bi-Lipschitz equivalent to a planar quasicircle.
\end{abstract}
% that a metric quasicircle has Assouad dimension strictly less than two \ifff it is bi-Lipschitz equivalent to a planar quasicircle.

%------------%
% dedication %
%------------%
%\dedicatory{Dedicated to ... ?}

%%======================%%
%%>>--START OF PAPER--<<%%
%%======================%%
\newcommand{\PV}{{{\small \em  preliminary version---please do not circulate}}}
\maketitle
%\tableofcontents
% BTintro.tex -- created 22 April 2009
%%======================================%%
\section{Introduction}  \label{S:Intro} %%
%%======================================%%
% edit history
%  18 aug 09 -- change proof of Arc Lemma, add doubling bs, other minor edits
%  17 jan 10 -- rewrite intro slightly
%  1,2 apr 10-- DH new stuff
%  12 oct 10 -- move fig 1

By definition, a \emph{metric quasicircle} is the \qsc\ image of the unit circle $\mfS^1$.  (See \rf{S:Prelims} for definitions and basic terminology.)  We exhibit a catalog that contains a \bl\ copy of each metric quasicircle.  This is a metric space analog of recent work by Steffen Rohde \cite{Rohde-qcircles-mod-bl}, so we briefly describe his result.  He constructed a collection $\mcR$ of snowflake type planar curves with the intriguing property that each planar quasicircle (the image of $\mfS^1$ under a global \qc\ self-\homeo\ of the plane) is \bl\ equivalent to some curve in $\mcR$.

Rohde's catalog is $\mcR:=\bigcup \mcR_p$, where $p\in[1/4,1/2)$ is a \emph{snowflake parameter}.  Each curve in $\mcR_p$ is built in a manner reminiscent of the construction of the von Koch snowflake.  Thus, each $R\in\mcR_p$ is the limit of a sequence $(R^n)$ of polygons where $R^{n+1}$ is obtained from $R^n$ by using the replacement rule illustrated in \rf{f:Rohde_snow}:  for each of the $4^n$ edges $E$ of $R^n$ we have two choices, either we replace $E$ with the four line segments obtained by dividing $E$ into four arcs of equal diameter, or we replace $E$  by a similarity copy of the polygonal arc $A_p$ pictured at the top right of \rf{f:Rohde_snow}.  In both cases $E$ is replaced by four new segments, each of these with diameter $(1/4)\diam(E)$ in the first case or with diameter $p\diam(E)$ in the second case.  The second type of replacement is done so that the ``tip'' of the replacement arc points into the exterior of $R^n$.  This iterative process starts with $R^1$ being the unit square, and the snowflake parameter, thus the polygonal arc $A_p$, is fixed throughout the construction.  See the discussion at the beginning of \rf{s:C} for more details.

The sequence $(R^n)$ of polygons converges, in the Hausdorff metric, to a planar quasicircle $R$ that we call a \emph{Rohde snowflake} constructed with snowflake parameter $p$.  Then $\mcR_p$ is the collection of all Rohde snowflakes that can be constructed with snowflake parameter $p$.

Rohde \cite[Theorem~1.1]{Rohde-qcircles-mod-bl} proved the following.
\begin{noname}
  A planar Jordan curve is a quasicircle \ifff it is \\ the image of some Rohde snowflake under a \bl \\ self-homeomorphism of the plane.
\end{noname}

%*****************%
\begin{figure}[t] % we could try this at bottom of page 2 or 3 ??
%*****************%
  \centering
  \begin{overpic}[width=12cm, tics=10]{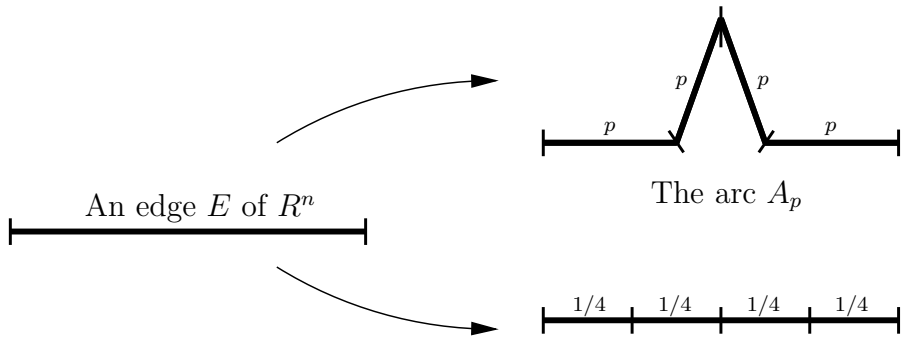}
  \put(9,13.5){An edge $E$ of $R^n$}
  \put(71.7,15){The arc $A_p$}
    \put(66.5,23){$\scriptstyle{p}$}
    \put(91,23){$\scriptstyle{p}$}
    \put(74.5,28){$\scriptstyle{p}$}
    \put(83.5,28){$\scriptstyle{p}$}
    \put(63,3){$\scriptstyle{1/4}$}
    \put(72.5,3){$\scriptstyle{1/4}$}
    \put(82.3,3){$\scriptstyle{1/4}$}
    \put(92.1,3){$\scriptstyle{1/4}$}
  \end{overpic}
  \caption{Construction of a Rohde-snowflake.}
  \label{f:Rohde_snow}
\end{figure}
%**************%

Thanks to a celebrated theorem of Ahlfors \cite{Ahlfors-QCrflxns}, there is a \emph{simple geometric criterion} that characterizes planar quasicircles: a planar Jordan curve $\Gam$ is a quasicircle if and only if it satisfies the \emph{bounded turning condition}, which means that there is a constant $C\geq 1$ such that for each pair of points $x,y$ on $\Gam$, the smaller diameter subarc $\Gam[x,y]$ of $\Gamma$ that joins $x,y$ satisfies
\begin{equation}  \label{eq:bt}
  \diam(\Gam[x,y])\le C\,|x-y| \,.  \tag{BT}
\end{equation}
We say $\Gamma$ is \emph{$C$-bounded turning} to emphasize the constant $C$.

Tukia and \Va\ \cite{TV-qs} introduced the notion of a \emph{\qsy} between metric spaces.  In this same paper they established the following metric space analog of Ahlfors' result.
\begin{noname}
A metric Jordan curve is a metric quasicircle \ifff it is both \bt\ and doubling (that is, of finite Assouad dimension).
\end{noname}

\medskip

Our catalog $\mcS$ of metric snowflake curves is a collection of metric circles $(\mfS^1,d)$ where the metrics $d$ are given in a simple way by specifying the diameter of each dyadic subarc of $\mfS^1$.  See \eqref{E:dist} and the end of \rf{s:DDF} for precise details.

Our catalog is $\mcS:=\bigcup \mcS_\sig$, and we also employ an auxiliary \emph{snowflake parameter} $\sig\in[1/2,1]$.  Each $(\mfS^1,d_\sig)$ in $\mcS_\sig$ has a metric $d_\sig$ that is obtained by the assignment of diameters to each dyadic subarc of $\mfS^1$.  As in Rohde's construction, at each step there are two choices: the diameter (\wrt $d_\sig$) of a given dyadic subarc is either one-half, or $\sig$, times the diameter of its parent subarc.

Each $(\mfS^1,d_\sig)$ is a \bt\ circle.  Moreover, when $\sig<1$, $(\mfS^1,d_\sig)$ has Assouad dimension $\alpha \leq \log 2/ \log(1/\sigma)<\infty$ (so, $2^{-1/\alpha}\le \sig<1$), hence $(\mfS^1,d_\sig)$ is doubling and thus a metric quasicircle; see \rf{L:dist}(e).  In fact, each collection $\mcS_\sig$ (with $\sig<1$) contains a \bl\ copy of every metric quasicircle with Assouad dimension strictly less than $\log(2)/\log(1/\sig)$.  In addition, the sub-catalog $\mcS_1$ contains a \bl\ copy of every \bt\ circle.

Here is our main result.

%%+++++++++++%%
\begin{thm*} %%
%%+++++++++++%%
  Let $\Gam$ be a metric Jordan curve.
  \begin{itemize}
    \item[(A)]  If $\,\Gam$ is \bt, then $\Gam$ is \bl\ equivalent to some curve in $\mcS_1$.
    \item[(B)]  If $\,\Gam$ is a metric quasicircle with Assouad dimension $\alf:=\dimA(\Gam)$ and $\sig\in(2^{-1/\alf},1)$, then $\Gam$ is \bl\ equivalent to a curve in $\mcS_\sig$.
    \item[(C)]  A metric quasicircle is \bl\ equivalent to a \emph{planar quasicircle} if and only if it has Assouad dimension strictly less than two.   % $\dimA(\Gam)<2$.
%    \item[(C)]  If $\,\Gam$ is a metric quasicircle, then it is \bl\ equivalent to a \emph{planar quasicircle} if and only if it has Assouad dimension below two.   % $\dimA(\Gam)<2$.
%    \item[(C)]  If $\,\Gam$ is a metric quasicircle, then it is \bl\ equivalent to a \emph{planar quasicircle} if and only if it has  $\dimA(\Gam)<2$.
  \end{itemize}
\end{thm*}
%%+++++++++++%%

\noindent
This result is quantitative in that the \bl\ constants depend only on the given data.  For example, if $\Gam$ is $C$-\bt, then the \bl\ constant in \rm(A) is
$$
  L=8\,C \max\{\diam(\Gam),\diam(\Gam)^{-1}\}.
$$
Minor modifications to our proofs reveal that the analogous results hold for \bt\ Jordan arcs and metric quasiarcs.

In addition, we explain how to recover Rohde's theorem from our result.  This provides an alternative proof of Rohde's result that avoids the technical construction of a ``uniform doubling measure'' appearing in \cite[Theorem~1.2]{Rohde-qcircles-mod-bl}.  In view of this, our argument somewhat simplifies the proof of Rohde's theorem.

We mention that Bonk, Heinonen, and Rohde have established a result that gives metric quasicircles as metric boundaries of certain metric disks; see \cite[Lemma~3.7]{BHR-cfml}.

\smallskip

The novel ideas in our approach include the following.  We make extensive use of the fact that every \bt\ metric space is \bl\ equivalent to its associated diameter distance space; see \rf{L:BT iff BL}.  In particular, this permits us to restrict attention to 1-\bt\ Jordan curves.  In this setting, the metrics are characterized, up to \bl\ equivalence, by knowledge of the diameters of certain subarcs, provided we have a sufficiently plentiful collection of subarcs; see \rf{L:dist}.  Finally, there is a straightforward way to build a \bl\ \homeo\ from one of our model curves onto such a metric Jordan curve; see \rf{P:subdiv_homeo} and \rf{L:d BL_to_d_D_variant2k}.

%\medskip

This document is organized as follows.  \rf{S:Prelims} contains preliminary information including background material on Assouad dimension (in \rf{s:assouad-dimension}) and on \qsc\ \homeo s (in \rf{s:QS}).   We prove a result about dividing an arc into subarcs of equal diameter (in \rf{s:dividing-arcs}) and (in \rf{s:shrinking}) give a useful tool for constructing \homeo s between Jordan curves.  We construct our dyadic models in \rf{s:DDF} and prove our Theorem in \rf{S:Proof}.
\tableofcontents        % to see organization
% BTprelims.tex -- created 6 April 2010
%%=========================================%%
\section{Preliminaries}  \label{S:Prelims} %%
%%=========================================%%
% edit history
%  6 April 2010 -- DHs new version

Here we set forth our (relatively standard) notation and terminology and present fundamental definitions and basic information.  First we provide some background on quasisymmetric maps, doubling, and bounded turning.  In \rf{s:IDD} we show that we can restrict attention to $1$-bounded turning circles.  In \rf{s:dividing-arcs} we prove that one can divide an arc into subarcs of equal diameter.  In \rf{s:shrinking} we establish a useful proposition for constructing \homeo s between Jordan arcs or curves.

%%%
%%%Our notation is, for the most part, relatively standard.  We write $C=C(a,\ldots)$ to indicate a constant $C$ which depends only on the parameters $a,\ldots$; the notation  $A\lex B$ means there exists a finite constant $c$ with $A\le c B$, and $A\simeq B$ means that both $A \lex B$ and $B\lex A$ hold.  Typically $a,b,c,C,K,\dots$ will be constants that depend on various parameters, and we try to make this as clear as possible often giving explicit values, however, at times $C$ will denote some constant whose value depends only on the data present and which may differ even on the same line of inequalities.
%%%
%%%For real numbers $a$ and $b$,
%%%$$
%%%a\wedge b := \min\{a,b\} \quad\text{and}\quad a\vee b := \max\{a,b\} \,.
%%%$$

%%---------------------------------------------------%%
\subsection{Basic Information}  \label{s:basic info} %%
%%---------------------------------------------------%%
%\newcommand{\ssm}{\!\!\sim}
For the record, $\mfN$ denotes the set of natural numbers, i.e., the positive integers.

We view the unit circle $\mfS^1$ as the unit interval with its endpoints identified; that is, $\mfS^1=[0,1]/\{0\!\!\sim\!\!1\}=[0,1]/\!\!\sim$ where $s\sim t$ \ifff either $s=t$ or $\{s,t\}=\{0,1\}$.  Then $\lam$ denotes the (normalized) arc-length metric on $\mfS^1$: for $s,t\in\mfS^1$ with say $0\le s\le t\le1$,
$$
  \lam(s,t) := \min\{t-s, 1-(t-s)\} \,.  %   \label{eq:euc_d_I}
$$

A (closed) \emph{Jordan curve} is the \homic\ image of the circle $\mfS^1$ and a \emph{metric Jordan curve} is a Jordan curve with a metric on it.  A \emph{Jordan arc} is the \homic\ image of the unit interval $[0,1]$ and a \emph{metric Jordan arc} is a Jordan arc with a metric on it.  Thus Jordan curves and arcs are non-degenerate compact spaces, where non-degenerate means not a single point.

Given distinct points $x,y$ on a metric Jordan curve $\Gam$, we write $\Gam[x,y]$ to denote the closure of the smaller diameter component of $\Gam\sm\{x,y\}$; when both components have the same size, we randomly pick one.  We often fix an orientation on $\Gam$, and then $[x,y]$ stands for the subarc of $\Gam$ that joins $x$ to $y$.

We note the following easy consequence of uniform continuity.

%:::::::::::::::::::::::::::::%
\begin{lma}  \label{L:finite} %
%:::::::::::::::::::::::::::::%
Let \,$\Gam$ be a metric Jordan curve or arc.  Then for each $\veps>0$, there are at most finitely many non-overlapping subarcs of \,$\Gam$ that all have diameter at least $\veps$. %::%
\end{lma} %::::::::::::::::%
%>>>>>>>>>>>>%
\begin{proof}%
%>>>>>>>>>>>>%    Can prolly delete this, but here it is....
Suppose $\Gam=\vphi(\mfS^1)$ for some \homeo\ $\vphi$.  Let $\veps>0$ be given.  Choose $\del>0$ so that for each subarc $I\subset\mfS^1$ with $\diam_\lam(I)<\del$ we have $\diam(\vphi(I))<\veps/2$.  Pick $N\in\mfN$ with $1/N<\del$.  Partition $\mfS^1$ into adjacent equal length subarcs $I_1,\dots,I_N$.

Let $A$ be a subarc of $\Gam$ with $\diam(A)\ge\veps$.  Then $A$ must contain at least one of the subarcs $\vphi(I_i)$.  Thus there are at most $N$ such subarcs $A$.

A similar argument applies when $\Gam$ is an arc.
%<<<<<<<<<<%
\end{proof}%
%<<<<<<<<<<%

Throughout this article we employ the Polish notation $\abs{x-y}$ for the distance between points $x,y$ in a metric space.
The \emph{bounded turning condition} (BT), also called \emph{Ahlfors' three point condition}, makes sense in any connected metric space: this holds whenever points can be joined by continua whose diameters are no larger than a fixed constant times the distance between the original points.  To be precise, given a constant $C\ge1$, we say that $X$ has the {\em $C$-bounded turning\/} property if each pair of points $x,y\in X$ can be joined by a continuum $\Gam[x,y]$ satisfying (BT).
% $\diam(A) \le C\,|x-y|$; we abbreviate this by declaring that $X$ is $C$-\bt.
The bounded turning condition has a venerable position in quasiconformal analysis; see for example \cite{TV-qs}, \cite{G-qdisks}, \cite{NV-John}, \cite{Tukia-bt} and the references therein.

A metric Jordan curve that is bounded turning is called a \emph{\bt\ circle}, or a \emph{$C$-\bt\ circle} if we wish to indicate the \bt\ constant $C$.

%---------------------------------------------------------%
\subsection{Assouad Dimension} \label{s:assouad-dimension}%  \label{sec:assouad-dimension}
%---------------------------------------------------------%
A metric space is \emph{doubling} if there is a number $N$ such that every subset of diameter $D$ has a cover that consists of at most $N$ subsets each having diameter at most $D/2$.  It follows that every set of diameter $D$ has a cover by (at most) $N^k$ sets each of diameter at most $D/2^k$.

The \emph{Assouad dimension} $\dimA(X)$ of a metric space $X$ is the infimum of all numbers $\alpha>0$ with the property that there exists a constant $C>0$ such that for all $D>0$, each subset of diameter $D$ has a cover consisting of at most $C \veps^{-\alpha}$ sets each of diameter at most $\veps D$.

An equivalent description can be given in terms of separated sets.  A subset $S\subset X$ is $r$-\emph{separated} provided it is non-degenerate, meaning  $\card(S)>1$, and for all distinct $x,y\in S$, $|x-y|\ge r$; in particular, $\diam(S)\ge r$.  Then  $\dimA(X)$ is the infimum of all numbers $\alpha>0$ with the property that there exists a constant $C>0$ such that for all $r>0$, each $r$-separated set $S\subset X$ has $\card(S)\le C (\diam(S)/r)^{\alpha}$.

Evidently, a metric space has finite Assouad dimension if and only if it is doubling.  The Assouad dimension was introduced by Assouad in \cite{Assouad-thesis} (see also \cite{Assouad-dim}).  A comprehensive overview is given in \cite{Luuk-ass-dim}.  The role of doubling spaces in the general theory of quasisymmetric maps is explained in \cite{Juha-analysis}.  The Assouad dimension of a space is a \bl\ invariant, and it is always at least the Hausdorff dimension.

\showinfo{
%:::::::::::::::::::::::::%
\begin{lma} \label{L:lma} %
Suppose $\dimA(X)<\beta$.  Then there is an $\veps_0\in(0,1]$ such that for all $\veps\in(0,\veps_0)$, the cardinality of any $\veps d$-separated set $S\subset X$ with $d=\diam(S)$ satisfies  $\card(S) < \veps^{-\beta}$.
\end{lma} %:::::::::::::::%
%>>>>>>>>>>>>%
\begin{proof}%
%>>>>>>>>>>>>%
Put $\alf:=\dimA(X)$ and $\gam:=(\alf+\beta)/2$.  Then $\gam>\dimA(X)$, so there exists a constant $C:=C(\gam)\ge2$ such that for all $d>0$, all $\veps\in(0,1]$, and each $\veps d$-separated $S\subset X$ with $d=\diam(S)$, $\card(S)\le C\, \veps^{-\gam}$.  Since $\beta-\gam=(\beta-\alf)/2>0$, there exists $t_0=t_0(\alf,\beta,C)>0$ such that for all $t>t_0$, $t^{\beta-\gam}>C$.
Let $\veps_0:=1/t_0$.  Then for all $\veps\in(0,\veps_0)$ and all $\veps d$-separated sets $S\subset X$ with $d=\diam(S)$,
$$
  \card(S) \le C\, \veps^{-\gam} < \veps^{\gam-\beta} \veps^{-\gam} =  \veps^{-\beta} \,.
$$
\flag{So, $\veps_0=\veps(\alf,\beta)$ where $\alf:=\dimA(X)$, but this is a bit hokey...  However, this kind of quantitativeness is needed to bound the BL constant in part (B).}
%<<<<<<<<<<%
\end{proof}%
%<<<<<<<<<<%
}

%%----------------------------------------------------%%
\subsection{Quasisymmetric Homeomorphisms} \label{s:QS} %%
%%----------------------------------------------------%%
A homeomorphism $X\overset{f}\to Y$ of metric spaces $X,Y$ is called a \emph{quasisymmetry} if there is a homeomorphism $\eta \colon
[0,\infty) \to [0,\infty)$ such that for all distinct $x,y,z \in X$ and $t\in [0,\infty)$,
\begin{equation*}
  \frac{\abs{x-y}}{\abs{x-z}} \leq t
  \quad\implies\quad
  \frac{\abs{f(x)-f(y)}}{\abs{f(x)-f(z)}} \leq \eta(t) .
\end{equation*}

This notion of \qsy\ was introduced by Tukia and \Va\ in \cite{TV-qs} where they also studied \emph{weak-quasisymmetries}.  A homeomorphism $f\colon X\to Y$ is a weak-\qsy\ if there is a constant $H\geq 1$, such that for all distinct $x,y,z\in X$,
\begin{equation*}
  \frac{\abs{x-y}}{\abs{x-z}} \leq 1
  \quad\implies\quad
  \frac{\abs{f(x)-f(y)}}{\abs{f(x)-f(z)}} \leq H.
\end{equation*}
Clearly every \qsy\ is a weak-\qsy.  Tukia and \Va\ proved that each weak-\qsy\ from a pseudo-convex space to a doubling space is a \qsy\ \cite[Theorem~2.15]{TV-qs}; Heinonen has a similar result for maps from a connected doubling space to a doubling space \cite[Theorem~10.19]{Juha-analysis}.  In particular, this holds for maps between Euclidean spaces.  However, a weak-\qsy\  may fail to be \qsc\ if the target space is not doubling, as illustrated by an example in the paper by Tukia and \Va.

As discussed in the Introduction, a \emph{metric quasicircle} is the quasisymmetric image of $\mfS^1$; thanks to work of Tukia and \Va, we know that these are precisely the doubling \bt\  circles.  Recently the second author \cite{Meyer-weakQS} established the following characterization of \bt\ circles.
\begin{noname}
  A metric Jordan curve is \bt\ if and only if it is a \emph{weak-quasisymmetric} image of the unit circle.
\end{noname}

%%-------------------------------------------%%
\subsection{Diameter Distance} \label{s:IDD} %%  \label{sec:inner-diam-dist}
%%-------------------------------------------%%
\renewcommand{\thefootnote}{\fnsymbol{footnote}}    % change footnote counter to get
\setcounter{footnote}{1}                            %    symbols instead of numbers

Here we show that we can always restrict attention to $1$-bounded turning circles.  More precisely, we show that any bounded turning circle is bi-Lipschitz equivalent to a $1$-bounded turning circle.  The relevant tool employed is the notion of \emph{diameter distance}\footnote{This is also called \emph{inner diameter distance}.} $\dia$ that is defined on any path connected metric space $(X,\ed)$ by
$$
  \dia(x,y):=\inf\{\diam(\gam) \mid \gam \;\text { a path in $X$ joining $x,y$ }\} \,.
$$
It is not hard to see that $\dia$ is a metric on $X$.
%%and that for all $x,y\in X$, $|x-y|\le\dia(x,y)$; thus the identity map $\id:(X,\dia)\to(X,\ed)$ is always 1-Lipschitz continuous.
Here are some additional properties of $\dia$. % diameter distance.

%:::::::::::::::::::::::::::::::%
\begin{lma} \label{L:BT iff BL} %
%:::::::::::::::::::::::::::::::%
Let $(\Gam,\ed)$ be a metric Jordan curve or a metric Jordan arc and let $\dia$ be the associated diameter distance.
\begin{enumerate}
  \item[\rm(a)]  \label{item:dia3}
    The $\dia$-diameter of any subarc $A$ of $\,\Gam$ equals its diameter \wrt the original metric on $X$; that is, $\diam_{\dia}(A) = \diam(A)$.
%%    \begin{equation*}
%%      \diam_{\dia}(A) = \diam(A).
%%    \end{equation*}

  \item[\rm(b)]  \label{item:dia4}
    For all points $x,y\in\Gam$, $\diam_{\dia}(\Gam[x,y])=\dia(x,y)$.  In particular, $(\Gam,\dia)$ is $1$-\bt.

  \item[\rm(c)]  \label{item:dia2}
    $(\Gam,\ed)$ is $C$-\bt\ \ifff the identity map $(\Gam,\dia)\overset{\id}\to(\Gam,\ed)$ is $C$-\bl.

\end{enumerate}
\end{lma} %:::::::::::::::::::::%
%>>>>>>>>>>>>%
\begin{proof}%
%>>>>>>>>>>>>%
To prove (a), first observe that for all $x,y\in \Gamma$, $|x-y|\le\dia(x,y)$, so $\diam(A)\le\diam_{\dia}(A)$.  Next, for all $x,y\in A$, $\dia(x,y)\le\diam(A)$, so $\diam_{\dia}(A)\le\diam(A)$.

Now (b) follows directly from (a) since
$$
  \dia(x,y)=\diam(\Gam[x,y])=\diam_{\dia}(\Gam[x,y]) \,.
$$

It remains to establish (c).  If $(\Gam,\ed)$ is $C$-\bt, then for all $x,y\in\Gam$
$$
  \dia(x,y)=\diam(\Gam[x,y])\le C\, |x-y| \le C\,\dia(x,y)
$$
so the identity map is $C$-\bl.  Conversely, if this map is $C$-\bl, then for all $x,y\in\Gam$
$$
  \diam(\Gam[x,y])=\diam_{\dia}(\Gam[x,y])=\dia(x,y)\le C\, |x-y|
$$
and therefore $(\Gam,\ed)$ is $C$-\bt.
%<<<<<<<<<<%
\end{proof}%
%<<<<<<<<<<%

We remark that in general the identity map $(X,\dia) \xrightarrow{\id}
(X,\ed)$ need not be a homeomorphism.  A simple example of this is the planar \emph{comb space}
$$
  X:=\lp [0,1]\times\{0\}\rp \cup \lp \{0\}\times[0,1] \rp \bigcup_{n=1}^\infty \lp \{1/n\}\times[0,1] \rp \subset\mfR^2
$$
equipped with Euclidean distance $\ed$.  If $z_n:=(1/n,1)$ and $a:=(0,1)$, then $|z_n-a|\to0$ as $n\to\infty$, whereas $\dia(z_n,a)\ge 1$ for all $n$.  Also, $(X,\ed)$ is compact but $(X,\dia)$ is not.
\showinfo{In fact, $\id$ will be a \homeo \ifff $(X,\ed)$ is locally path connected.  This is in my notes, but perhaps need not be mentioned here.}

%%-----------------------------------------------------%%
\subsection{Division of Arcs}  \label{s:dividing-arcs} %% {sec:dividing-arcs} %%
%%-----------------------------------------------------%%
Here we prove that any metric Jordan arc can be divided into any given number of subarcs each having exactly the same diameter.

The problem of finding points on a metric Jordan arc such that consecutive points are at the same distance is non-trivial.  In 1930 Menger gave a proof \cite[p.\ 487]{Menger}, that is short, simple and natural, but wrong.  It was proved for arcs in Euclidean space in \cite{Alt-Beer}, and in the general case (indeed in more generality) in \cite[Theorem 3]{Schoenberg}; see also \cite{V-dividing}.

For the case at hand, i.e., for bounded turning arcs, it suffices to find adjacent subarcs that have equal diameter.  We give the following elementary proof for this problem.

%::::::::::::::::::::::::::::::::%
\begin{prop} \label{P:equi_diam} %\label{lem:equi_diam}
%::::::::::::::::::::::::::::::::%
  Let $A$ be a metric Jordan arc and $N\geq 2$ an integer. Then we can divide $A$ into $N$ subarcs of equal diameter.
\end{prop} %:::::::::::::::%
%>>>>>>>>>>>>%
\begin{proof}%
%>>>>>>>>>>>>%
We may assume that $A$ is the unit interval $[0,1]$ equipped with some metric $d$.  We claim that there are points $0=s_0 < s_1 < \dots < s_{N-1} < s_N=1$ such that
\begin{equation*}
  \diam[s_0,s_1] = \diam [s_1,s_2] = \dots = \diam[s_{N-1},s_N]
\end{equation*}
where $\diam$ denotes diameter with respect to the metric $d$.  When $N=2$ this follows by applying the Intermediate Value Theorem to the function $[0,1]\ni s\mapsto \diam [0,s] - \diam[s,1]$.

According to \rf{L:BT iff BL}(a), we may replace $d$ by its associated diameter distance; thus we may assume from the start that for any $[s,t]\subset [0,1]$
\begin{equation} \label{eq:d_diam}
  d(s,t) = \diam [s,t] \,.
\end{equation}

\smallskip
Next, we modify $d$ to get a metric $d_\veps$ that is \emph{strictly increasing} in the sense that
\begin{equation} \label{eq:de_strictly_inc}
  [s,t] \subsetneq [s',t'] \subset[0,1] \implies d_\veps(s,t) <  d_\veps(s',t') \,.
\end{equation}
The crucial point here is the \emph{strict} inequality, which need not hold in general.
To this end, fix $\veps>0$ and for all $s,t\in [0,1]$ set
\begin{equation*}
  d_\veps(s,t) := d(s,t) +\veps\abs{t-s} \,.
\end{equation*}
Then from \eqref{eq:d_diam} it follows that
\begin{equation*}
  \diam_\veps [s,t] = \diam[s,t] + \veps\abs{t-s} =
  d_\veps(s,t) \,,
\end{equation*}
where $\diam_\veps$ denotes diameter with respect to
$d_\veps$.  This immediately implies \eqref{eq:de_strictly_inc}.

\smallskip
We now show that $[0,1]$ can be divided into $N$ subintervals of
equal $d_\veps$-diameter.
Consider the compact set $S:=\{\mathbf{s}=(s_1,\dots, s_{N-1}) \mid 0\leq s_1 \leq
\dots \leq s_{N-1}\leq 1\}$. Set $s_0:=0, s_N:=1$. The function
$\varphi \colon S\to \mfR$ defined by
\begin{equation*}
  \varphi(\mathbf{s}) := \max_{0\leq i\leq N-1} \diam_\veps [s_i,s_{i+1}] -
  \min_{0\leq j\leq N-1} \diam_\veps [s_j,s_{j+1}]
\end{equation*}
assumes a minimum on $S$. If this minimum is zero, we are
done. Otherwise, there are adjacent intervals $[s_{i-1}, s_{i}],
[s_{i}, s_{i+1}]$ that have different $d_\veps$-diameter. Using the
Intermediate Value Theorem as before, we can find $s'_i\in [s_{i-1},
s_{i+1}]$ such that $\diam_\veps [s_{i-1},s'_i] =
\diam_\veps[s'_i, s_{i+1}]$. Then from (\ref{eq:de_strictly_inc}) it
follows that
\begin{align*}
  \min_{0\leq j < N} \diam_\veps [s_j,s_{j+1}]
  & <
  \diam_\veps [s_{i-1},s'_i] \\
  & =
  \diam_\veps[s'_i, s_{i+1}]
  <
  \max_{0\leq i < N} \diam_\veps [s_i,s_{i+1}].
\end{align*}
Applying this procedure to all subintervals of maximal
$d_\veps$-diameter we obtain a strictly smaller
minimum for the function $\varphi$, which is impossible. Thus the minimum must be zero,
and so we can subdivide $[0,1]$ into $N$ subintervals of equal
$d_\veps$-diameter.

\smallskip
Consider now a sequence $\veps_n\searrow 0$, as $n\to
\infty$. Let $s_1^n< \dots < s_{N-1}^n$ be the points that divide
$[0,1]$ into $N$
subintervals of equal diameter with respect to $d_{\veps_n}$. We
can assume that for all $1\leq j < N$, all points $s^n_j$ converge
to $s_j$ as $n\to\infty$. It follows that
for all $1\leq i,j < N$,
\begin{equation*}
  \diam[s_i,s_{i+1}] = \lim_{n\to\infty} \diam_{\veps_n} [s^n_i,s^n_{i+1}] =
  \lim_{n\to\infty} \diam_{\veps_n} [s^n_j,s^n_{j+1}] = \diam[s_j, s_{j+1}]
\end{equation*}
as desired.
%<<<<<<<<<<%
\end{proof}%
%<<<<<<<<<<%

The previous Lemma is also true for metric Jordan curves $\Gamma$.  In this case we are free to choose any point in $\Gamma$ to be an endpoint of one of the subarcs.

%%-------------------------------------------------------%%
\subsection{Shrinking Subdivisions}  \label{s:shrinking} %%
%%-------------------------------------------------------%%
Here we present a useful tool for constructing homeomorphisms between Jordan curves; see \rf{P:subdiv_homeo}.

We begin with some terminology.
%%  We call $\mcA$ a \emph{decomposition of $\Gam$ into subarcs} if $\mcA$ is a finite family of compact non-overlapping non-degenerate subarcs of $\Gam$ whose union equals $\Gam$; thus $\Gam=\bigcup_{A\in\mcA} A$, each $A\in\mcA$ is an arc with positive diameter, and distinct elements of $\mcA$ have disjoint interiors.  Next, $(\mcA^n)_1^\infty$ is a \emph{subdivision sequence for $\Gam$} provided each $\mcA^n$ is a decomposition of $\Gam$ into subarcs and each $\mcA^{n+1}$ is a subdivision of $\mcA^n$, meaning that for every arc $A\in\mcA^{n+1}$ there is a (unique) arc in $\mcA^n$ that contains $A$.  Finally, $(\mcA^n)_1^\infty$ is a \emph{shrinking subdivision for $\Gam$} if it is a subdivision sequence that \emph{shrinks} in the sense that
%%\begin{equation*}
%%  \max_{A\in\mcA^n} \diam(A) \to 0 \text{ as } n\to \infty.
%%\end{equation*}
Let $\Gam$ be a metric Jordan curve or arc.
A sequence $(\mcA^n)_1^\infty$ is a \emph{shrinking subdivision} for $\Gam$ provided:
\begin{itemize}
  \item
  Each $\mcA^n$ is a finite \emph{decomposition} of $\Gam$ into compact arcs.
  Thus each $\mcA^n$ is a finite set of non-overlapping non-degenerate compact
  subarcs of $\Gam$ that cover $\Gam$.  (Here non-overlapping means disjoint
  interiors and non-degenerate means not a single point.)

  \item
  Each $\mcA^{n+1}$ is a \emph{subdivision} of $\mcA^n$; i.e., for each
  arc $A$ in $\mcA^{n+1}$ there is a (unique) arc in $\mcA^n$, called the \emph{parent of $A$}, that contains $A$.

  \item
  The subdivisions \emph{shrink}, meaning that $\ds  \max_{A\in\mcA^n} \diam(A) \to 0 \text{ as } n\to \infty$.
%    \begin{equation*}
%      \max_{A\in\mcA^n} \diam(A) \to 0 \text{ as } n\to \infty.
%    \end{equation*}
\end{itemize}

%%By connecting each arc to its parent, we can view $\{\Gam\}\cup\bigcup_1^\infty\mcA^n$ as the vertex of a rooted tree; here $\Gam$ is the root and %%the parent of each $A\in\mcA^1$.  In this connection, we use the following elementary fact on various subtrees.
%%
%%\theoremstyle{plain}
%%\newtheorem*{KsL}{K\H{o}nig's Lemma}
%%\begin{KsL}
%%  A rooted tree with infinitely many vertices, each of finite degree, contains an infinite simple path.
%%\end{KsL}

Assume $(\mcA^n)_1^\infty$ is a shrinking subdivision for $\Gam$.  We call $(A^n)_1^\infty$ a \emph{descendant sequence} if $A^1\supset A^2\supset\dots$ and $A^n\in\mcA^n$ for all $n\in\mfN$; thus each $A^n$ is the parent of $A^{n+1}$.  Note that for any descendant sequence $(A^n)_1^\infty$, $\bigcap_1^\infty A^n$ is a single point.  Also, for each point $x\in\Gam$, there exists a descendant sequence $(A_x^n)_1^\infty$ with $\{x\}=\bigcap_1^\infty A_x^n$; such a descendant sequence need not be unique, but there can be at most two such sequences.

Shrinking subdivisions are useful for constructing homeomorphisms between metric Jordan curves; see \rf{s:A}, \rf{s:B}, \rf{s:C}.
%%  It is convenient to formulate this procedure as a general principle as follows.

%:::::::::::::::::::::::::::::::::::%
\begin{prop} \label{P:subdiv_homeo} %
%:::::::::::::::::::::::::::::::::::%
Let $\mfA$ and $\mfB$ both be metric Jordan curves or metric Jordan arcs.  Suppose $(\mcA^n)_1^\infty$ and $(\mcB^n)_1^\infty$ are shrinking subdivisions for $\mfA$ and $\mfB$ respectively.  Assume these subdivisions are \emph{combinatorially equivalent}, meaning that for each $n\in\mfN$ there are bijective maps $\Phi^n\colon \mcA^n \to \mcB^n$ such that for all $A, \tilde{A}\in \mcA^n$ and $A_0\in \mcA^{n+1}$
\begin{align*}
  A \cap \tilde{A} = \emptyset &\quad\iff\quad \Phi^n(A) \cap \Phi^n(\tilde{A}) =\emptyset\,,  \\
  A_0 \subset A &\quad\iff\quad \Phi^{n+1}(A_0) \subset  \Phi^{n}(A) \,.
\end{align*}
Then the sequence $(\Phi^n)_1^\infty$ induces a \emph{homeomorphism} $\mfA\overset{\vphi} \to \mfB$ with the property that
$$
  \text{for all $n\in\mfN$ and all }\; A\in\mcA^n \,,\quad \varphi(A) =\Phi^n(A) \,.
$$
\end{prop} %:::::::::::::::%
%>>>>>>>>>>>>%
\begin{proof}%
%>>>>>>>>>>>>%
Let $a\in\mfA$ and select a descendant sequence $(A^n)_1^\infty$ with $\{a\}=\bigcap_1^\infty A^n$.  Setting $B^n:=\Phi^n(A^n)$ we obtain a descendant sequence $(B^n)_1^\infty$ with, say, $\{b\}:=\bigcap_1^\infty B^n$.  Suppose $(\tilde{A}^n)_1^\infty$ is a second descendant sequence with $\{a\}=\bigcap_1^\infty \tilde{A}^n$.  Let $\tilde{B}^n:=\Phi^n(\tilde{A}^n)$ and $\{\tilde{b}\}=\bigcap_1^\infty \tilde{B}^n$.  Since $A^n\cap\tilde{A}^n\ne\emptyset$, $B^n\cap\tilde{B}^n\ne\emptyset$ and therefore
$$
  \abs{b-\tilde{b}} \le \diam(B^n) + \diam(\tilde{B}^n ) \to 0 \quad\text{as $n\to\infty$} \,.
$$
Thus $\tilde{b}=b$ and so there is a well defined map $\vphi:\mfA\to\mfB$ given by setting $\vphi(a):=b$.

Two distinct points $a_1,a_2\in\mfA$ lie in disjoint arcs $A_1,A_2\in\mcA^n$, for sufficiently large $n\in\mfN$, and then $\vphi(A_1)\cap\vphi(A_2)=\emptyset$, so $\vphi(a_1)\ne\vphi(a_2)$ verifying that $\vphi$ is injective.

Given $b\in\mfB$ and a descendant sequence $(B^n)_1^\infty$ with $\{b\}=\bigcap_1^\infty B^n$,  $A^n:=(\Phi^n)^{-1}(B^n)$ defines a descendant sequence $(A^n)_1^\infty$ with, say, $\{a\}:=\bigcap_1^\infty A^n$, and then $\vphi(a)=b$.  Thus $\vphi$ is surjective.

%%Using the bijections $\Psi^n:=(\Phi^n)^{-1}:\mcB^n\to\mcA^n$ in exactly the same manner, we define a map $\psi:\mfB\to\mfA$ with the property that
%%$$
%%  \text{for all $x\in\mfA$ and all $y\in\mfB$}\,, \quad \psi\comp\vphi(x)=x \;\text{ and }\; \vphi\comp\psi(y)=y \,.
%%$$
%%It follows that $\vphi$ and $\psi$ are bijections, and also that for each $A\in\mcA^n$, $\vphi(A)=\Phi^n(A)$.

Let $\veps>0$ be arbitrary.  Fix an $n\in\mfN$ such that $\max\{\diam(B) \mid B\in \mcB^n\}< \veps/2$.  Let $\delta:=\min\{\dist(A_1,A_2) \mid A_1,A_2\in\mcA^n \,;\; A_1\cap A_2=\emptyset\}$.  Suppose $a_1,a_2\in\mcA$ with $|a_1-a_2|<\del$.  Pick $A_k\in\mcA^n$ with $a_k\in\A_k$.  The definition of $\del$ ensures that $A_1\cap A_2\ne\emptyset$.  Therefore,
$$
  |\vphi(a_1)-\vphi(a_2)|\le \diam(\vphi(A_1)) + \diam(\vphi(A_2)) \le \veps
$$
and so $\varphi$ is (uniformly) continuous and hence a \homeo.
%<<<<<<<<<<%
\end{proof}%
%<<<<<<<<<<%
   % basic info + subarcs
% BTdyadics.tex -- created 6 April 2010
%%=================================================================%%
\section{Dyadic Subarcs and Diameter Functions}  \label{S:Dyadics} %%
%%=================================================================%%
% edit history

Here we give precise definitions of our model curves, i.e., our model circles.  These are given by defining metrics on $\mfS^1$.  Since we can restrict attention to 1-\bt\ circles (thanks to \rf{L:BT iff BL}(b,c)), it suffices to only know the diameters of certain subarcs, provided we have a sufficiently plentiful collection of subarcs; for this purpose we use the dyadic subarcs described in \rf{s:DS}.  We introduce the notion of a dyadic diameter function in \rf{s:DDF}; these provide a simple method for constructing metrics on $\mfS^1$.  Then in \rf{s:BL equiv} we establish a convenient way to detect when two such metrics are \bl\ equivalent, and also when a given metric Jordan curve is \bl\ equivalent to $\mfS^1$ with such a metric.

%%----------------------------------------%%
\subsection{Dyadic Subarcs}  \label{s:DS} %%
%%----------------------------------------%%
With our convention that $\mfS^1=[0,1]/\{0\!\!\sim\!\!1\}$, the $n^{\rm th}$-generation dyadic subarcs of $\mfS^1$ (obtained by dividing $\mfS^1$ into $2^n$ subarcs of equal diameter) are the subarcs of the form
\begin{gather*}
  I_k^n:=[k/2^n,(k+1)/2^n] \quad\text{where $k\in\{0,1,\dots,2^n-1\}$}. \\
  \intertext{Noting that $I^0:=I_0^0:=\mfS^1$, we define}
  \mcI^{n}:=\{I_k^{n} \mid k\in\{0,1,\dots,2^n-1\} \} \quad \text{and then} \quad  \mcI:=\bigcup_{n=0}^\infty  \mcI^{n} .
\end{gather*}
%%\begin{align*}
%%  & I_k^n:=[k/2^n,(k+1)/2^n] \quad\text{where $0\le k<2^n$}. \\
%%  \intertext{Noting that $I^0:=I_0^0:=\mfS^1$, we define}
%%  & \text{for each}\; n\in\mfN\,, \;\quad \mcI^{(n)}:=\{I_k^{(n)} \mid k\in[0,2^n)\cap\mfN \}, \\
%%  & \text{and then} \quad  \mcI:=\bigcup_{n=0}^\infty  \mcI^{n} .
%%\end{align*}
Each dyadic subarc $I^n\in \mcI^n$ contains exactly two
$I^{n+1},\tilde{I}^{n+1}\in \mcI^{n+1}$ that we call the \emph{children} of $I^n$, and then $I^n$ is the
\emph{parent} of each of $I^{n+1}, \tilde{I}^{n+1}$.

It is convenient to introduce some terminology. Often, we denote the \emph{children} or \emph{sibling} or \emph{parent} of a generic $I\in\mcI$ by
$$
  I_0 \,, I_1 \;\text{ or }\; \tilde{I} \;\text{ or }\; \hat{I}
$$
respectively; implicit in the use of the latter two notations is the requirement that $I\ne\mfS^1$.

Clearly, $(\mcI^n)_1^\infty$ is a shrinking subdivision for $\mfS^1$ in the sense of \rf{s:shrinking}.
Recall too that a sequence $(I^{n})_{n=0}^\infty$ of dyadic subarcs $I^n\in\mcI^n$ is a \emph{descendant sequence} provided $I^{0}\supset I^{1}\supset I^{2}\supset\dots$; that is, for each $n$, $I^{n+1}$ is a child of $I^{n}$.  We note that for each $x\in\mfS^1$ there is a descendant sequence $(I_x^{n})_{n=0}^\infty$ with $\{x\}=\bigcap_{n=0}^\infty I_x^{n}$; such a sequence is unique unless $x$ is a dyadic endpoint in which case there are exactly two such sequences.

By connecting each arc to its parent, we can view $\mcI$ as the vertex set of a rooted binary tree.  In this connection, we use the following elementary fact on various subtrees.

\theoremstyle{plain}
\newtheorem*{KsL}{K\H{o}nig's Lemma}
\begin{KsL}
  A rooted tree with infinitely many vertices, each of finite degree, contains an infinite simple path.
\end{KsL}

In our setting this means that each infinite subtree contains a descendant sequence.

\medskip

In the proof of part (B) of our Theorem it will be convenient to ``do $m$ steps at once''.  This means that instead of dividing an arc into two subarcs, we will divide it into $2^m$ subarcs.  With this in mind, we also consider the family $\mcJ$ of all $2^m$-adic subarcs; thus
$$
  \mcJ := \bigcup_{n=0}^\infty  \mcJ^{n} \quad\text{where $\mcJ^{n} = \mcI^{mn}$}.
$$
Each $\mcJ^{n}$ contains the $2^{mn}$ subarcs of the form
$J_k^n:=[k/2^{mn}, (k+1)/2^{mn}]$ in $\mcI^{mn}$ with $k\in\{0,1,\dots,2^{mn}-1\}$.  Each such arc $J^n$ has
$2^m$ children, i.e., arcs $J^{n+1}\in \mcI^{m(n+1)}$, all of which are contained in $J^n$.
%%We note the following useful fact:
%%\begin{equation}  \label{E:I vs J}
%%  \forall\; I\in\mcI\sm\mcJ\,, \; \exists\, J,J_0\in\mcJ \;\text{ with $J_0$ a child of $J$ relative to $\mcJ$ and }\; J_0\subset I\subset J.
%%\end{equation}
%%\flag{Evidently, when $m=1$ we have $\mcJ=\mcI$.  In the proof of part (C ) of our Theorem we will work with 4-adic subarcs, so $m=2$, and we write %%$\mcK$ instead of $\mcJ$, so $\mcK:=\bigcup_{n=0}^\infty  \mcK^{n}$ where $\mcK^{n} = \mcI^{2n}$.  Say ??}

%%----------------------------------------------------%%
\subsection{Dyadic Diameter Functions}  \label{s:DDF} %%
%%----------------------------------------------------%%
A dyadic diameter function $\Del$ assigns a diameter $\Delta(I)$ to each dyadic subarc $I\in \mcI$. More precisely, we call $\Del:\mcI\to(0,1]$  a \emph{dyadic diameter function constructed using the snowflake parameter $\sig\in[1/2,1]$} provided $\Del(\mfS^1)=1$ and
\begin{gather*}
  \forall\; I\in\mcI \;, \quad\text{either}\quad \Del(I_0)=\Del(I_1):=\half\,\Del(I) \quad\text{or}\quad \Del(I_0)=\Del(I_1):=\sig\,\Del(I)  \\
  \intertext{where $I_0, I_1$ are the two children of $I$.  When $\sig=1$, we also require}
  \lim_{n\to\infty} \max \left\{ \Del(I) \mid I\in\mcI^n \right\} = 0 \,.
\end{gather*}
If $\sig<1$, this latter condition is automatically true.  The snowflake parameter $\sig$ is kept fixed throughout the construction.

%%% general DDFs
%%%
%%%A dyadic diameter function $\Del$ assigns a diameter $\Delta(I)$ to each dyadic subarc $I\in \mcI$. More precisely, $\Del:\mcI\to(0,1]$ is a \emph{dyadic diameter function} provided
%%%\begin{gather*}
%%%  \Del(\mfS^1)=1 \;\text{ and }\;  \lim_{n\to\infty} \max_k \Del(I_k^{(n)}) = 0  \;\text{ and}  \\
%%%  \forall\; I\in\mcI \;, \quad \half \Del(I) \le \Del(I_0)=\Del(I_1) \le \Del(I)
%%%\end{gather*}
%%%where $I_0, I_1$ are the two children of $I$.
%%%
%%%For example, given a \emph{snowflake parameter} $\sig\in[1/2,1]$ we can define $\mcI\overset{\Del}\to(0,\infty)$ by setting $\Del(\mfS^1)=1$ and then requiring that for all $n\in\mfN$ and each $I\in\mcI^n$
%%%$$
%%%  \text{either}\quad \Del(I_0):=\Del(I_1):=\half\,\Del(I) \quad\text{or}\quad \Del(I_0):=\Del(I_1):=\sig\,\Del(I).
%%%$$
%%%Here $\sig$ is fixed throughout the construction and when $\sig=1$ we assume that choices are made so that the maximum of $\Del(I)$ over all $I\in\mcI^n$ tends to zero.  We describe such a dyadic diameter function $\Del$ by saying that $\Del$ is \emph{constructed using the snowflake parameter $\sig$}.

\showinfo{
By taking $\sig=1/2$ we recover the (normalized) Euclidean arc-length metric $\lam$.  With $\sig=1$ we obtain a \emph{simple} dyadic diameter function  that satisfies
\begin{equation}\label{E:simple}
  \forall\; I\in\mcI \,, \quad \Del(I)/\Del(I_0)=\Del(I)/\Del(I_1) \in \{1,2\}.
\end{equation}
}

%%%We emphasize that whenever we discuss dyadic diameter functions, we {\bf always} assume that the above conditions hold.  See \rf{L:Itozero} below.

Each dyadic diameter function $\Del$ produces a distance function $d=d_\Del$ on $\mfS^1$ defined by
\begin{equation}  \label{E:dist}
      d(x,y)=d_\Del(x,y):=\inf  \sum_{k=1}^N \Del(I_k)
\end{equation}
where the infimum is taken over all \emph{$xy$-chains} $I_1,\dots,I_N$ in $\mcI$; thus $x$ and $y$ lie in $I_1\cup\dots\cup I_N$, each $I_k$ belongs to $\mcI$, and for all $2\le k \le N$, $I_{k-1}\cap I_k\ne\emptyset$.
%
%%\begin{equation}  \label{E:dist}
%%      d(x,y)=d_\Del(x,y):=\inf \left\{ \sum_{k=1}^N \Del(I_k) \mid I_1,\dots,I_N \;\text{ an $xy$-chain from $\mcI$} \right\} \,.
%%\end{equation}
%%Here $I_1,\dots,I_N$ is an \emph{$xy$-chain from $\mcI$} provided $x,y\in I_1\cup\dots\cup I_N$, each $I_k$ belongs to $\mcI$, and for all $2\le k \le N$, $I_{k-1}\cap I_k\ne\emptyset$.
%
%%Note that to determine the infimum in \eqref{E:dist} we may assume that $x\in I_1$, $y\in I_N$, and that the $I_n$ are adjacent, intersect only at %%their endpoints, and are maximal \wrt Euclidean diameter; indeed, if $J=I_k\cup\dots\cup I_M\in\mcI$ for some $1\le k\le M\le N$, then the properties %%of $D$ ensure that $D(J)\le \sum_{m=k}^M D(I_m)$.

Now we present various properties of this metric.  Our `diameter function' terminology is motivated by item (d) below.

%::::::::::::::::::::::::::%
\begin{lma} \label{L:dist} %
%::::::::::::::::::::::::::%
Let $\mcI\overset{\Del}\to(0,\infty)$ be a dyadic diameter function and define $d:=d_\Del$ as in \eqref{E:dist}.  Then:
\begin{enumerate}
  \item[\rm(a)]  $d$ is a metric on $\mfS^1$.  % i.e.  $\mfS^1_d:=(\mfS^1^1,d)$ is a metric Jordan curve (i.e., ),
  \item[\rm(b)]  The identity map $\id:(\mfS^1,d)\to(\mfS^1,\lam)$ is a 1-Lipschitz \homeo; recall that $\lam$ is the normalized length metric on $\mfS^1$; see \rf{s:basic info}.
  \item[\rm(c)]  $(\mfS^1,d)$ is 1-\bt\ (so $d$ is its own diameter distance).
  \item[\rm(d)]  The diameter (with respect to $d$) of each dyadic subarc is given by $\Delta$; i.e., for  all $n\in\mfN$ and all $I\in\mcI^{n}$, $\diam_d(I)=\Del(I)$.
  \item[\rm(e)]  If $\Del$ is constructed using a snowflake parameter $\sig\in[1/2,1)$, then the Assouad dimension of $(\mfS^1,d)$ is at most $\log 2/\log(1/\sig)$.  Equality holds for the ``extremal model'' where we take $\Del(I_0) = \Del(I_1) = \sig\, \Delta(I)$ for both children $I_0,I_1$ of each $I\in\mcI$.
\end{enumerate}
\end{lma} %::::::::::::::::%
\noindent
%%In particular, thanks to the Tukia-\Va\ characterization, we see that in case (e) the metric curve $(\mfS^1,d)$ is a metric quasicircle.  Note that %%metric circles constructed with the snowflake parameter $\sig=1$ may not be doubling.
%>>>>>>>>>>>>%
\begin{proof}%
%>>>>>>>>>>>>%
(a) It is clear that $d$ is non-negative, symmetric, and satisfies the triangle inequality.  Given $x\in \mfS^1$ and $n\in \mfN$, let $I^n_x\in \mcI^n$ be a dyadic subarc containing $x$.  Since $(I^n_x)_1^\infty$ is an $xx$-chain, $d(x,x)\leq \Delta(I^n_x) \to 0$ (as $n\to \infty$), so $d(x,x)=0$. Since $\Delta(I^n)\geq 2^{-n}=\diam_\lam(I^n)$, it follows that $d(x,y)\geq \lambda(x,y)$. Thus $d(x,y)= 0$ \ifff $x=y$.

\smallskip\noindent(b)
This follows from \rf{P:subdiv_homeo} and the penultimate sentence in the proof of (a).

\smallskip\noindent(d)
Fix $I\in\mcI^{n}$ with $n\geq 1$.  For all points $x,y\in
I$, $I$ is an $xy$-chain, so $d(x,y)\le \Delta(I)$  and thus $\diam_d(I)\le
\Delta(I)$.  The opposite inequality follows from the observation that any
chain joining the endpoints of $I$ must cover either $I$ or its sibling $\tilde{I}$.
%%%$J$ where $J\in \mcI^n$ is the (unique) sibling of $I$ (i.e., $I,J$ are distinct
%%%and have the same parent $I^{n-1}\in \mcI^{n-1}$).

\smallskip\noindent(c)
To demonstrate that $(\mfS^1,d)$ is $1$-\bt, fix distinct points
$x,y\in\mfS^1$.  Let $[x,y]$ and $[y,x]$ be the two closed
arcs on $\mfS^1$ between $x,y$ (i.e., the closures of
the components of $\mfS^1\sm\{x,y\}$). Assume that $\diam_d([x,y])
\leq \diam_d([y,x])$.  Next let $I_1,\dots,I_N$ be
any $xy$-chain.  Then $I_1\cup\dots\cup I_N\supset A$, where either
$A=[x,y]$ or $A=[y,x]$, so $\diam_d([x,y])\leq \diam_d(A)$.

For any $a,b\in A$, $I_1,\dots, I_N$ is an $ab$-chain; therefore
$$
  d(a,b) \le \sum_{n=1}^N \Delta(I_n) \,,\;\text{ and thus }\; \diam_d([x,y]) \le \diam_d(A) \le \sum_{n=1}^N \Delta(I_n) \,.
$$
Taking the infimum over all such $xy$-chains $I_1,\dots,I_N$ yields $$\diam_d([x,y])\le d(x,y)\,.$$

%\smallskip
\noindent(e)
First, suppose $\Del$ is constructed using a snowflake parameter $\sig\in[1/2,1)$.  Let $\alf:=\log 2/\log(1/\sig)$, so $\sig^{-\alf}=2$.  Fix an arbitrary $\veps\in(0,1]$.  Choose $n\in \mfN$ so that $\sig^n < \veps \leq \sig^{n-1}$. Consider a dyadic subarc $I^n\in\mcI^n$. Then
$\diam_d(I^n) =\Delta(I^n) \leq \sig^n< \veps$.

Now let $A$ be any $\veps$-separated set in $(\mfS^1,d)$. Then $A$ contains at most one point in each dyadic subarc $I^n\in \mcI^n$. Thus
$$
  \card(A) \leq 2^n = \lp \sig^{-\alf} \rp^n = \sig^{-\alpha} \lp \sig^{n-1} \rp^{-\alpha} \le 2\, \veps^{-\alf} \,.
$$
It follows that the Assouad dimension of $(\mfS^1,d)$ is at most $\alpha$; see \rf{s:assouad-dimension}.

\smallskip
Finally, consider the dyadic diameter function given by setting $\Delta(I^{n+1}) :=\sig\, \Delta(I^n)$ (for each child $I^{n+1}\in \mcI^{n+1}$ of every $I^n\in \mcI^{n}$) and its corresponding metric  $d=d_\Delta$.  Then for each $n\in\mfN$, the set $A^n:=\{k/2^n \mid 0\le k < 2^n\}$ of $n^{\rm th}$-generation endpoints is $\sig^n$-separated in $(\mfS^1,d)$. Assume constants $C>0, \alpha>0$ are given so that the number of $\veps$-separated points is at most $C\veps^{-\alpha}$.  Taking $\veps:=\sig^n$ we obtain
\begin{equation*}
  C\veps^{-\alpha} = C (\sig^n)^{-\alpha} = C (\sig^{-\alpha})^n \geq \card(A^n) = 2^n \,,\;\text{ so }\; \alpha\geq \frac{\log 2}{\log(1/\sig)}\,.  \qedhere
\end{equation*}
%<<<<<<<<<<%
\end{proof}%
%<<<<<<<<<<%

\showinfo{The following is useful for checking the requirement that the diameters tend to zero.

%:::::::::::::::::::::::::::::%
\begin{lma} \label{L:Itozero} %
%:::::::::::::::::::::::::::::%
Suppose the function $\mcI\overset{\Delta}\to(0,\infty)$ satisfies $\Del(\hat{I})\ge\Del(I)$ for each $I\in\mcI\sm\{\mfS^1\}$.  Then \tfae:
\begin{enumerate}
  \item[\rm(a)]  $\lim_{n\to\infty} \max \{\Delta(I^{n}) \mid I^n \in \mcI^n\}= 0$.
  \item[\rm(b)]  For every sequence $(I^{n})_0^\infty$ (with $I^{n}\in\mcI^{n}$),
                 $\lim_{n\to\infty} \Delta(I^{n}) = 0$.
  \item[\rm(c)]  For every descendant sequence
                 $(I^{n})_0^\infty$, $\lim_{n\to\infty} \Delta(I^{n}) = 0$.
\end{enumerate}
\end{lma} %::::::::::::::::%
%>>>>>>>>>>>>%
\begin{proof}%
%>>>>>>>>>>>>%
  The implications (a)$\iff$(b)$\implies$(c) are clear.
  Assume (a) is false. Then there is an $\veps>0$ such that the set
  $\mcI_\veps:= \{I^n\in \mcI \mid \Delta(I^n)\geq \veps\}$ is
  infinite. Note that if $I^n$ is in $\mcI_\veps$, then its parent
  is contained in $\mcI_\veps$ as well. Thus $\mcI_\veps$ naturally forms
  a graph, where we connect each $I^n\in \mcI_\veps$ to
  its parent. This is an infinite tree. By K\H{o}nig's Lemma
  there is a descendant
  sequence $I^0\supset I^1\supset \dots$ with $\Delta(I^n) \geq
  \veps$ for all $n\in \mfN$. Thus (c) is false, so (c)$\implies$(a).
%<<<<<<<<<<%
\end{proof}%
%<<<<<<<<<<%
}

\medskip

Given $\sig\in[1/2,1]$, we let $\mcS_\sig$ be the collection of all metric circles $(\mfS^1,d)$, where the metric $d=d_\Del$ is defined as in \eqref{E:dist} and $\Del:\mcI\to(0,1]$ is any dyadic diameter function constructed using the snowflake parameter $\sig$.  Then
$$
  \mcS:=\bigcup_{\sig\in[1/2,1]} \mcS_\sig
$$
is our catalog of snowflake type metric circles.
Thanks to the Tukia-\Va\ characterization, \rf{L:dist}(c,e) imply that for $\sig\in[1/2,1)$, each curve in $\mcS_\sig$ is a metric quasicircle.

The curves in $\mcS_1$ are \bt\ circles, but need not be metric quasicircles since they may fail to be doubling. There is a simple test for doubling that we give below in Lemma~\ref{L:dblg}.
%%%However, it is straightforward to show that a metric circle in $\mcS_1$ (defined via some dyadic diameter function $\Del$) is doubling, i.e., is a %%%metric quasicircle, if and only if there exists an $n_0\in \mfN$ such that for all $I^n\in \mathcal{I}^n$ and any descendant $I^{n+n_0} \in %%%\mathcal{I}^{n+n_0}$, $\Delta(I^{n+n_0}) \leq \half\Delta(I^n)$.  Roughly, we get doubling \ifff diameters are always at least halved after a fixed %%%number of steps.
%%%
%%%\showinfo{Here is the old Lemma about when $\Del$ is doubling.
%%%As noted above, there are dyadic diameter functions whose associated metrics are not doubling.
% Because of the importance of doubling, we mention that there is an
% easy way to detect this for metric circles in $\mcS_1$.

%%--------------------------------------------------------%%
\subsection{$2^m$-adic Diameter Functions}  \label{s:2DF} %%
%%--------------------------------------------------------%%
We also require $2^m$-adic diameter functions; recall (see the end of \rf{s:DS}) that $\mcJ$ denotes the family of $2^m$-adic subarcs of $\mfS^1$.  We call $\Del:\mcJ\to(0,1]$  a \emph{$2^m$-adic diameter function constructed using the snowflake parameter $\tau\in[1/2^m,1]$} provided $\Del(\mfS^1)=1$ and
\begin{align*}
  \forall\; J\in\mcJ \;, \quad\text{either}\quad & \Del(J_0)=\Del(J_1)=\dots=\Del(J_{2^m-1}):=\frac{1}{2^m}\,\Del(J)  \\
                         \quad\text{or}\quad & \Del(J_0)=\Del(J_1)=\dots=\Del(J_{2^m-1}):=\tau\,\Del(J)
\end{align*}
where $J_0,\dots,J_{2^m-1}$ are the children of $J$. The
snowflake parameter $\tau$ is fixed throughout the construction. If
$\tau=1$, we also require
$$
  \lim_{n\to\infty} \max \left\{ \Del(J) \mid J\in\mcJ^n \right\} = 0 \,.
$$
When $\tau<1$ this latter condition is automatically true.

\smallskip
Just as for dyadic diameter functions, each $2^m$-adic diameter function $\Del$ has an associated distance function $d_\Del$ defined as in \eqref{E:dist} but now we only consider $xy$-chains chosen from $\mcJ$.  \rf{L:dist} remains valid for $2^m$-adic diameter functions; however, in part (e) we must take $\sig=\tau^{1/m}$, where the $2^m$-adic diameter function is constructed using the snowflake parameter $\tau\in[1/2^m,1]$.

\medskip
% There is a natural correspondence between dyadic and $2^m$-adic
% diameter functions that we now describe.
We note the following
useful fact.   For each dyadic arc $I\in\mcI$, there exist $2^m$-adic
arcs $J^n\in\mcJ^n$ and $J^{n+1}\in\mcJ^{n+1}$ such that
\begin{equation}  \label{E:I vs J}
  J^{n+1} \subset I\subset J^n \,.
\end{equation}
% Clearly each dyadic diameter function $\Del:\mcI\to(0,1]$, constructed using a snowflake parameter $\sig\in[1/2,1]$, restricts to a $2^m$-adic diameter function on $\mcJ$, constructed using $\tau:=\sig^m\in[1/2^m,1]$.

Each $2^m$-adic diameter function $\Del:\mcJ\to(0,1]$, with snowflake parameter $\tau$, has a natural extension to a dyadic diameter function $\Del:\mcI\to(0,1]$, with snowflake parameter $\sigma:=\tau^{1/m}$, that is defined as follows.  Fix a subarc $J^n\in\mcJ$ and let $J^{n+1}\subset J^n$ be any child of $J^n$.  Let $J^n=:I^{mn}\supset I^{mn+1}\supset \dots \supset I^{m(n+1)}:=J^{n+1}$ be the finite descendant sequence from $\mcI$ determined by $J^{n+1}$ and $J^n$.  Set
\begin{gather*}
  \rho:=[\Del(J^{n+1})/\Del(J^n)]^{1/m} \quad\text{(so, $\rho\in\{1/2,\tau^{1/m}\}$)} \\
  \intertext{and for each $i\in\{0,1,\dots,m\}$ define}
  \Del(I^{mn+i}) := \rho^i \Del(J^n) \,.
\end{gather*}
In view of \eqref{E:I vs J}, this procedure defines $\Del(I)$ for each $I\in\mcI$. Note that $\Delta(I^{mn+0})= \Delta(J^n)$ and $\Delta(I^{mn+m})= \Delta(J^{n+1})$, so $\Del:\mcI\to(0,1]$ is an extension of $\Delta\colon \mcJ\to (0,1]$. Clearly this extension is a dyadic diameter function constructed with the snowflake parameter $\sig=\tau^{1/m}$.

%%Note that if $\tilde{I}$ is the sibling of $I=I^{mn+i}$ (relative to $\mcI$), then $I^{mn+i-1}$ is the parent of both $I$ and $\tilde{I}$, and there %%is some child $J_1$ of $J$ (relative to $\mcJ$) with $J_1\subset\tilde{I}\subset J$.  Then $J_0$ and $J_1$ are siblings, so we see that %%$\Del(\tilde{I})=\Del(I)$ as required.

% \medskip
% Note that two different dyadic diameter functions may restrict to the same $2^m$-adic diameter function.  However, any two such dyadic diameter functions are \bl\ equivalent.  This is because the extension of a $2^m$-adic diameter function to a dyadic diameter function is $2^m$-\bl\ equivalent to the original diameter function.  Here is a more precise statement.

%::::::::::::::::::::::::::::::::::%
\begin{lma} \label{L:DDF to 2^mDF} %
%::::::::::::::::::::::::::::::::::%
Let $\Delta\colon\mcJ\to(0,1]$ be a $2^m$-dyadic diameter function that has been extended to all dyadic intervals, i.e., to a dyadic diameter function $\Delta\colon \mcI\to (0,1]$, as described above. Let $d_\mcI$ and $d_\mcJ$ be the metrics defined via $\Delta|_{\mcI}$ and $\Delta|_{\mcJ}$ respectively, meaning by \eqref{E:dist} and using chains from $\mcI$ and $\mcJ$ respectively.
% associated with $\Del$ and $\Del\!\mid_\mcJ$ respectively.  % (so, defined via \eqref{E:dist} using chains from $\mcI$ and $\mcJ$ respectively).
Then for all $x,y\in\mfS^1$,
$$
  \frac1{2^m} d_\mcJ(x,y) \le d_\mcI(x,y) \le d_\mcJ(x,y) \,.
$$
\end{lma} %::::::::::::::::::::::::%
%>>>>>>>>>>>>%
\begin{proof}%
%>>>>>>>>>>>>%
The right-hand inequality holds because there are more $xy$-chains available when we use subarcs from $\mcI$.  To prove the left-hand inequality, let $I_1,\dots,I_N$ be an $xy$-chain from $\mcI$.  Now use \eqref{E:I vs J} to get a corresponding $xy$-chain $J_1,\dots,J_N$ from $\mcJ$  and with $J_k'\subset I_k \subset J_k$ where $J_k'$ is some child of $J_k$.  Then for each $k$
$$
  \Del(I_k) \ge \Del(J'_k) \ge 2^{-m} \Del(J_k)\,, \quad\text{so}\quad
  d_\mcJ(x,y) \le \sum_{k=1}^N \Del(J_k) \le 2^m \sum_{k=1}^N \Del(I_k) \,.
$$
%%\begin{gather*}
%%  \Del(I_k) \ge \Del(J'_k) \ge \frac1{2^m} \Del(J_k) \,.  \\
%%  \intertext{Thus}
%%  d_\mcJ(x,y) \le \sum_{k=1}^N \Del(J_k) \le 2^m \sum_{k=1}^N \Del(I_k) \,.
%%\end{gather*}
Taking an infimum gives $d_\mcJ(x,y)\le 2^m d_\mcI(x,y)$.
%<<<<<<<<<<%
\end{proof}%
%<<<<<<<<<<%

The previous lemma and prior discussion reveal that in order to prove that a given metric circle $(\Gam,\ed)$ is \bl\ equivalent to a curve in $\mcS_\sig$, it is sufficient to construct a $2^m$-adic model circle (with snowflake parameter $\tau= \sigma^m$) that is \bl\ equivalent to $(\Gam,\ed)$; this will yield a dyadic model circle (with snowflake parameter $\sigma$) \bl\ equivalent to $(\Gam,\ed)$.

%%By the previous lemma and prior discussion it will be enough to
%%construct a $2^m$-adic model (with snowflake parameter $\tau= \sigma^m$)
%%\bl\ equivalent to a given metric circle $(\Gamma,\abs{\cdot})$. This will yield the desired
%%dyadic model (with snowflake parameter $\sigma$) \bl\ equivalent to
%%$(\Gamma,\abs{\cdot})$.
%%% come into play in the proof
%%% of our Theorem as this says that \bl\ equivalence to some $2^m$-adic
%%% model circle automatically gives \bl\ equivalence to a dyadic model
%%% circle.  That is, in order to prove that some given metric
%%% quasicircle, or some given \bt\ circle, is \bl\ equivalent to a curve
%%% in $\mcS_\sig$, it is sufficient to construct a $2^m$-adic diameter
%%% function $\Del$, using the snowflake parameter $\tau:=\sig^{1/m}$, so
%%% that the given metric circle is \bl\ equivalent to $(\mfS^1,d_\Del)$.

\begin{rmk}\label{rmk:4m4adic}
  Rohde's construction is based on $4$-adic arcs rather than dyadic arcs.  Results similar to the above also hold in this case.  Namely each $4^m$-adic diameter function $\bigcup_k\mcI^{4^{mk}}\to(0,1]$, with snowflake parameter $\tau$ in $[1/4^m,1]$, has an extension to a $4$-adic diameter function with snowflake parameter $\sig:=\tau^{1/m}\in [1/4,1]$.  The analog of \rf{L:DDF to 2^mDF} holds: the metrics constructed from these two diameter functions are \bl\ equivalent.
\end{rmk}

%%--------------------------------------------------------%%
\subsection{Bi-Lipschitz Equivalence}  \label{s:BL equiv} %%  \label{sec:bi-lipsch-equiv}
%%--------------------------------------------------------%%

Let $(\Gam,\ed)$ be a \bt\ circle and $(\mfS^1,d_\Del)$ be a model circle where $\Del$ is some dyadic diameter function.  In the following we
show that to prove bi-Lipschitz equivalence of $(\Gam,\ed)$ and $(\mfS^1,d_\Del)$, it is enough to show bi-Lipschitz equivalence for dyadic
subarcs.  More precisely, we establish the following result.

%:::::::::::::::::::::::::::::::::::::::::::%
\begin{lma} \label{L:d BL_to_d_D_variant2k} %  \label{L:d BL_to_d_D_variant2k} %
%:::::::::::::::::::::::::::::::::::::::::::%
Let $(\Gam,\ed)$ be a $C$-\bt\ circle and $d=d_\Delta$ a metric on
$\mfS^1$ defined via a $2^m$-adic diameter function $\Delta$.  % (as in \eqref{E:dist} but using chains in $\mcJ$).
Let $\varphi\colon \mfS^1 \to \Gamma$ be a homeomorphism.
Suppose there exists a constant $K\geq 1$ such that for all $J\in\mcJ$,
$$
  K^{-1}\diam(\varphi(J)) \le \Delta(J) \le K\, \diam(\varphi(J)) \,.
$$
Then $(\mfS^1,d)\overset{\vphi} \to (\Gamma,\ed)$ is $L$-bi-Lipschitz where $L:=2^{m+1}C\,K$.
\end{lma} %:::::::::::::::%

Before proving this lemma (see \ref{proof}), we first give a simple way to estimate the diameter of an arc in terms of the diameters of dyadic subarcs.

%::::::::::::::::::::::::::::::::::::::%
\begin{lma} \label{L:arc lemma_2k_var} %
%::::::::::::::::::::::::::::::::::::::%
Let $\mcJ\overset{\Del}\to(0,1]$ be a $2^m$-adic diameter function with associated metric $d=d_\Delta$.
% (defined by \eqref{E:dist} but using chains from $\mcJ$).
For each arc $A\subset\mfS^1$, define
\begin{equation*}  \label{eq:def_delta}
  \D(A)=\D_\Del(A):= \max\{\Delta(I) \mid I\subset A, I \in \mcJ\}.
\end{equation*}
Then for all arcs $A\subset\mfS^1$,
\begin{equation*}
  \D(A)\leq \diam_d (A)\leq 2^{m+1}\D(A).
\end{equation*}
In fact, there are $2^m$-adic arcs $I,J\in \mcJ$ such that $I\cup J\subset A \subset \hat{I} \cup \hat{J}$, $\Delta(I)=\D(A)$, and either $I=J$ or $\hat{I},\hat{J}$ are adjacent.  Here $\hat{I},\hat{J}\in \mcJ$ are the parents of $I,J$ relative to $\mcJ$. %::::::::::::::::::%
\end{lma} %:::::::::%  Note that we could have  \Del(\hat{J}) > \Del(\hat{I})  !
%>>>>>>>>>>>>%
\begin{proof}%
%>>>>>>>>>>>>%
Let $A$ be a subarc of $\mfS^1$.
Suppose we have verified the existence of the described $2^m$-adic arcs $I,J\in \mcJ$.  Then
\begin{align*}
  \D(A) &=\Del(I)=\diam_d(I) \le \diam_d(A) \le \diam_d(\hat{I}\cup\hat{J})   \\
        &\le \diam_d(\hat{I})+\diam_d(\hat{J}) = \Del(\hat{I})+\Del(\hat{J})  \\
        &\le 2^m[\Del(I)+\Del(J)] \le 2^{m+1} \Del(I) = 2^{m+1} \D(A) \,.
\end{align*}
%%\begin{align*}
%%  \D(A) &=\Del(I)=\diam_d(I) \le \diam_d(A)  \\
%%        &\le \diam_d(\hat{I}\cup\hat{J}) \le \diam_d(\hat{I})+\diam_d(\hat{J})  \\
%%        &= \Del(\hat{I})+\Del(\hat{J}) \le 2^m[\Del(I)+\Del(J)]  \\
%%        &\le 2^{m+1} \Del(I) = 2^{m+1} \D(A) \,.
%%\end{align*}
%%\begin{align*}
%%  \D(A) &=\Del(I)=\diam_d(I) \le \diam_d(A) \le \diam_d(\hat{I}\cup\hat{J})  \\
%%        &\le \diam_d(\hat{I})+\diam_d(\hat{J}) = \Del(\hat{I})+\Del(\hat{J})   \\
%%        &\le 2^m[\Del(I)+\Del(J)]\le 2^{m+1} \Del(I) = 2^{m+1} \D(A) \,.
%%\end{align*}
Thus it suffices to exhibit such $I$ and $J$.

\smallskip
Suppose $\mcF\subset\mcJ$ is some family of $2^m$-adic arcs (e.g.,
defined by certain properties).  We say that an arc $I^n\in\mcJ^n$ is
\emph{maximal \wrt $\mcF$} provided $I^n\in\mcF$ and for all $J^l\in
\mcJ^l$ with
$J^l\in\mcF$, either $\Del(J^l)<\Del(I^n)$ or
$$
  \Del(J^l)=\Del(I^n) \quad\text{and }\; l\ge n \,.
$$
Thus $I^n$ is the ``largest'' arc in $\mcF$, and when there are
several such large arcs, ``seniority wins''.  Note that the parent of
such a maximal $I^n$ will not belong to $\mcF$.

\smallskip
Now assume $A$ is the oriented arc $[a,b]\subset\mfS^1=[0,1]/\!\!\sim$
with $0<a<b<1$.  Pick $I=I^n\in \mcJ$ so that $I\subset A$,
$\Del(I)=\D(A)$, and such that $I$ is maximal among all such arcs.
Let $\hat{I}\supset I$ be the $\mcJ$-parent of $I$.   If $A\subset
\hat{I}$, then upon setting $J:=I$  we are done.

\smallskip
Assume  that $A\not\subset \hat{I}$.  The maximality of $I$
ensures that one endpoint of $\hat{I}$, without loss of generality the
left endpoint, is not contained in $A$. Let $y$ be the right endpoint
of $\hat{I}$.  Then $[a,y]\subset \hat{I}$.

Now consider subarcs $J\in\mcJ$ that lie in $A$ and to the right of
$y$, and select the largest of these.  More precisely, let
$J=J^l\in\mcJ$ be the maximal $2^m$-adic subarc that contains $y$ as
its left endpoint and is contained in $[y,b]$.  Note that the
maximality of $I$ implies that %% $\Del(J)\le\Del(I)$ and
\begin{equation}\label{E:J vs I}
    \text{either}\quad l \ge n  \qquad\text{or}\quad \Del(J)<\Del(I) \,.
\end{equation}

Consider the parent $\hat{J}$ of $J$.  We claim that $\hat{J}$ contains a point to the right of $b$, and then since $A=[a,y]\cup[y,b]\subset\hat{I}\cup\hat{J}$, we are done.  If $\hat{J}$ did not contain a point to the right of $b$, then it would have to contain a point to the left of $y$, but as we now show this would lead to a contradiction.

So, suppose $\hat{J}$ contains a point to the left of $y$.  Then in particular, $y$ is an interior point of $\hat{J}$.  Since $y$ is an endpoint of $\hat{I}$, we cannot have $\hat{I}\supset\hat{J}$ nor $\hat{I}=\hat{J}$, and therefore $\hat{I}\subsetneq\hat{J}$.  This implies that $n>l$.  However, it also implies that some $2^m$-adic sibling $\tilde{J}$ of $J$ satisfies $\tilde{J}\supset\hat{I}$, and therefore $\Del(I)\le\Del(\hat{I})\le\Del(\tilde{J})=\Del(J)$.  In view of \eqref{E:J vs I}, one of these last two implications does not hold, so $\hat{J}$ cannot contain a point to the left of $y$.
%<<<<<<<<<<%
\end{proof}%
\begin{pf}{Proof of \rf{L:d BL_to_d_D_variant2k}}  \label{proof}%
%>>>>>>>>>>>>>>>>>>>>>>>>>>>>>>>>>>>>>>>>>>>>>>>>>>>>>>>>>>>>>>>%
An appeal to \rf{L:BT iff BL}(b,c) permits us to assume that $(\Gam,\ed)$ is 1-\bt.  Write $\Gam[x,y]$ for the smaller diameter subarc joining points $x,y$ on $\Gam$; so $\abs{x-y}=\diam(\Gam[x,y])$.  Fix points $s,t$ on $\mfS^1$ and put $x:=\vphi(s), y:=\vphi(t)$.   Let $[s,t], [t,s]$ be the two arcs in $\mfS^1$ joining $s,t$ and assume that $\diam_d([t,s])\geq \diam_d([s,t])=d(s,t)$.

\smallskip
First we show that $\abs{x-y}\le2^{m+1}K\,d(s,t)$.  Using \rf{L:arc lemma_2k_var} we select $2^m$-adic subarcs $I,J\in \mcJ$ with $I\cup J\subset [s,t] \subset \hat{I} \cup \hat{J}$, $\hat{I}\cap\hat{J}\ne\emptyset$, and
$$
  \Del(J)\le \Delta(I)=\D([s,t])\le \diam_d([s,t])=d(s,t) \,.
$$
Here $\hat{I},\hat{J}\in \mcJ$ are the parents of $I,J$ relative to $\mcJ$.  Then
\begin{align*}
  \abs{x-y} &=\diam(\Gam[x,y]) = \min\{\diam(\vphi[s,t]),\diam(\vphi[t,s]) \} \le \diam(\vphi[s,t]) \\
            &\le \diam(\vphi(\hat{I} \cup \hat{J})) \le K[\Del(\hat{I}) + \Del( \hat{J} ) ] \le 2^m K [\Del({I}) + \Del({J} ) ] \\
            &\le 2^{m+1} K \, \Del(I) \le 2^{m+1} K \, d(s,t) \,.
\end{align*}

%\smallskip
Next we show that $d(s,t)\le2^{m+1}K\,\abs{x-y}$.  Let $A$ be the subarc of $\mfS^1$---either $A=[s,t]$ or $A=[t,s]$---with $\vphi(A)=\Gam[x,y]$.  Again we use \rf{L:arc lemma_2k_var} to pick a subarc $I\in \mcJ$ with $I\subset A$ and $\Delta(I)=\D(A)$.  Then $\vphi(I)\subset\vphi(A)=\Gam[x,y]$, so
\begin{align*}
  d(s,t)&\le\diam_d(A)\le 2^{m+1}\D(A)=2^{m+1}\Del(I) \le 2^{m+1}K\,\diam(\vphi(I)) \\
        &\le 2^{m+1}K\,\diam(\Gam[x,y])=2^{m+1}K\,\abs{x-y} \,.
\end{align*}

\vspace*{-5mm}
%<<<<<<<%
\end{pf}%
%<<<<<<<%

\medskip
We end this subsection with a criterion that describes when a metric circle in $\mcS_1$ is doubling. Roughly speaking, we get doubling \ifff diameters are always at least halved after a fixed number of steps.

%::::::::::::::::::::::::::%
\begin{lma} \label{L:dblg} %
Let $\mcI\overset{\Del}\to(0,1]$  be a dyadic diameter function with
snowflake parameter $\sigma=1$ and define $d:=d_\Del$ as in
\eqref{E:dist}.  Then $(\mfS^1,d)$ is doubling \ifff there exists an
$n_0\in\mfN$ such that
$$
  \forall\; n\in\mfN \,, \forall\; I^{n} \,, \forall\; I^{n+n_0}\subset I^{n} \,, \quad  \Del(I^{n+n_0}) \le \half\, \Del(I^{n}) \,.
$$
\end{lma} %::::::::::::::::%
%>>>>>>>>>>>>%
\begin{proof}%
%>>>>>>>>>>>>%
Suppose $(\mfS^1,d)$ is doubling.  Then there are constants $C\ge1$ and $\alf\ge1$ such that for each $r$-separated set $E$ in $(\mfS^1,d)$,
$$
  \card(E) \le C \lp \diam_d(E)/r \rp^\alf \,.
$$
Let $I:=I^{n}\in \mcI^n$ be given.  Suppose $(I^{m})_{m=n}^{n+k}$ is a
descendant sequence with $\Del(I^{m})\ge r:=\half \Del(I)$ for all
$m\in\{n,n+1,\dots,n+k\}$. Let $E$ be the set of endpoints of all the
subarcs $I^n, \dots, I^{n+k}$.  To see that $E$ is $r$-separated, let $e_1,e_2$ be
two distinct points in $E$.  We can assume that $e_1$ is an endpoint
of some $I^i$ and $e_2\in I^j\subset I^i$, where $n\leq i< j \leq
n+k$, and that $I^j$ does not contain $e_1$ but $I^{j-1}$ does.  Then
the sibling $\tilde{I}^{j}$ of $I^j$ separates $e_1$ and $e_2$. Thus
$d(e_1,e_2)\geq \Delta(\tilde{I}^{j})= \Delta(I^j)\geq r$. 

Now $\diam_d(E)=\diam_d(I)=\Del(I)=2\,r$, so by doubling
$$
  k \le \card(E) \le C \lp \diam_d(E)/r \rp^\alf = 2^\alf C \,.
$$
Therefore $n_0:= \lceil 2^{\alpha}C\rceil +1$ is the desired number.

\medskip
% Conversely, suppose there is such an $n_0\in\mfN$. Let $I:=I^{n}$ be
% a given dyadic subarc.  Let $J_1, J_2, \dots, J_{2^{n_0}}$ be the
% dyadic subarcs in $\mcI^{n+n_0}$ that lie in $I$.  These cover $I$ and
% each $J_k$ has
% $$
%   \diam_d(J_k)=\Del(J_k)\le\half \Del(I)=\half \diam_d(I) \,,
% $$
% so
Conversely, suppose there is such an $n_0\in\mfN$.  Let $A\subset \mfS^1$ be any arc. Let $I\in\mcI^n$, $J\in \mcI^m$  be dyadic subarcs with parents $\hat{I}\in \mcI^{n-1}$, $\hat{J}\in \mcI^{m-1}$ as in \rf{L:arc lemma_2k_var}; thus $I\cup J \subset A \subset \hat{I}\cup \hat{J}$.  Let $I_1,\dots,
I_{2^{n_0+1}}\in \mcI^{n+n_0}$, $J_1,\dots, J_{2^{n_0+1}}\in \mcI^{m+n_0}$ be the dyadic subarcs contained in $\hat{I}$ and $\hat{J}$ respectively. Then for all $1\leq k \leq 2^{n_0+1}$
\begin{equation*}
  \diam_d(I_k)=\Del(I_k) \leq \half \diam_d(I) \leq \diam_d(A)
\end{equation*}
and similarly $\diam_d(J_k)\leq(1/2)\diam_d(A)$. Thus we obtain the doubling condition with $N:=2^{n_0+2}$.
%<<<<<<<<<<%
\end{proof}%
\section{Proof of the Main Theorem}  \label{S:Proof} %%
%%===============================================%%
Here we establish parts (A), (B), (C) of the Theorem stated in the Introduction; see \rf{s:A}, \rf{s:B}, \rf{s:C} respectively.  In addition, in \rf{s:C} we explain how to recover Rohde's Theorem. % and we provide a detailed proof that Rohde snowflakes are quasicircles.

\showinfo{If the subsection counters are changed to 1,2,3, then we should change our ``parts'' to (1),(2),(3).}

Recall from \rf{s:DDF} that $\mcS_\sig$ is the collection of all
metric circles $(\mfS^1,d_\sig)$ where the metrics $d_\sig=d_{\Del}$
are defined as in \eqref{E:dist} and $\Del:\mcI\to(0,1]$ is any dyadic
diameter function constructed using the snowflake parameter
$\sig\in[1/2,1]$.  Recall too that for $\sigma\in [1/2,1)$ each curve in $\mcS_\sig$ is a
metric quasicircle that has Assouad dimension at most
$\log2/\log(1/\sig)$; see \rf{L:dist}(c,e). 

%%As we shall see, up to \bl\ equivalence, the families $\mcS_\sig$ increase (\wrt containment) as $\sig$ increases, so $\mcS_1$ `contains' all the %%others.  More precisely, for each $\sig$ and $\tau$ with $1/2\le\tau<\sig\le1$, $\mcS_\sig$ contains a \bl\ copy of each curve in $\mcS_\tau$.  This %%is a consequence of our proofs of parts (A) and (B).

\medskip
For the remainder of this section, $(\Gam,\ed)$ is a \bt\ circle.  Our three proofs share the following common theme: We define an appropriate shrinking subdivision for $\Gam$ and then appeal to \rf{P:subdiv_homeo} and \rf{L:d BL_to_d_D_variant2k} to obtain the necessary \bl\ \homeo s.  In each case this involves constructing a dyadic diameter function $\Del$ using some snowflake parameter.
%, and then $(\mfS,d_\Del)$ is \bl\ equivalent to $(\Gam,\ed)$.

\smallskip
To start, we fix an orientation on $\Gamma$.  All subarcs inherit this orientation, and $[a,b]$ denotes the oriented subarc of $\,\Gam$ with endpoints $a,b$.  Next, an appeal to \rf{L:BT iff BL}(b,c) permits us to replace $\ed$ with its associated diameter distance thereby obtaining a \bl\ equivalent 1-\bt\ circle; the \bl\ constant for this change of metric equals the original \bt\ constant.  Thus we may, and do, assume that $(\Gam,\ed)$ is 1-\bt.  This means that
\begin{equation*} \label{eq:diamgam_gam}
  \diam([a,b]) = \abs{a-b} \quad\text{whenever}\quad  \diam([a,b])\leq \diam(\Gamma\setminus[a,b]) \,.
\end{equation*}
We also assume that $\diam(\Gamma)=1$; this involves another \bl\ change of metric with \bl\ constant $\max\{\diam(\Gam),\diam(\Gam)^{-1}\}$.

%%-------------------------------------%%
\subsection{Proof of (A)}  \label{s:A} %%
%%-------------------------------------%%
We assume $(\Gam,\ed)$ is 1-\bt\ with $\diam(\Gam)=1$; it need not be doubling.  We construct a dyadic diameter function $\Del$ on $\mcI$, using the snowflake parameter $\sig=1$, so that $(\Gam,\ed)$ is \bl\ equivalent to $(\mfS,d_\Del)$.

First, we divide $\Gamma$ into two arcs $A^1_0, A^1_1$ that both have
diameter one.  Then we inductively divide each arc into two subarcs of
equal diameter.  Appealing to \rf{P:equi_diam}, we divide each $A^n_i$
into two subarcs $A^{n+1}_{2i}, A^{n+1}_{2i+1}$ of equal diameter.
This defines subarcs $A^{n}_k$ for each $k\in\{0,1,\dots,2^n-1\}$ and
all $n\in \mfN$.  Here we label so that the $A^n_k$ are ordered
successively along $\Gamma$ with the initial point of $A^n_0$ the same
for all $n\in \mfN$. 

\showinfo{We could avoid using \rf{P:equi_diam} by selecting midpoints; a \emph{midpoint of an arc} is any point on it that is equi-distant from its endpoints.}

We claim that $\lim_{n\to\infty} \max_k \diam(A^n_k) =0$.  For suppose this does not hold.  Then there is an $\veps>0$ such that the set $\mathbf{\Gamma}_\veps:=  \{A^n_k \mid \diam(A^n_k)\geq \veps\}$ is infinite.  Noting that each parent of an arc in $\mathbf{\Gam}_\veps$ also belongs to $\mathbf{\Gam}_\veps$, we may appeal to K\H{o}nig's Lemma to obtain a descendent sequence $\mfS^1=:A^0\supset A^1\supset A^2\supset \dots$ (where $A^n=A^n_{k_n}$ is some arc in $\mathbf{\Gamma}_\veps$).
By construction $A^n$ is divided into two subarcs $A^{n+1}$ and $B^{n+1}$ of equal diameter, so $\diam(B^{n+1})\geq \veps$.  Then $\{B^1,B^2,\dots\}$ is an infinite collection of non-overlapping subarcs of $\Gam$ each with diameter at least $\veps$.  This contradiction to \rf{L:finite} implies that our claim must hold

\smallskip
By setting $\mcA^n:=\{A^n_k \mid k\in\{0,1,\dots,2^n-1\} \}$ (for each $n\in \mfN$) we obtain a shrinking subdivision $(\mcA^n)_1^\infty$ for $\Gam$; see \rf{s:shrinking}.  In fact, $(\mcI^n)_1^\infty$ and $(\mcA^n)_1^\infty$  are combinatorially equivalent shrinking subdivisions, and thus by \rf{P:subdiv_homeo} there is an induced homeomorphism $\varphi\colon \mfS^1 \to \Gamma$ with $\vphi(I_k^n)=A_k^n$ for all $n\in\mfN$ and all $k\in\{0,1,\dots,2^n-1\}$.

\smallskip
It remains to construct a dyadic diameter function $\Delta$ using the snowflake parameter $\sig=1$ and so that $\Del$ also satisfies the following: for all $n\in\mfN$,
\begin{equation}  \label{eq:DeltadiamGam}
  \text{for all $k\in\{0,1,\dots,2^n-1\}$}\;, \quad
  \frac{1}{2}\, \Delta(I^n_k) \leq \diam(A^n_k) \leq 2 \, \Delta(I^n_k) \,.
\end{equation}
Having accomplished this task, we can appeal to \rf{L:d BL_to_d_D_variant2k} (with $C=1$, $m=1$, $K=2$) to assert that $\varphi:(\mfS^1,d_\Del)\to(\Gam,\ed)$ is $8$-\bl.

\smallskip
We start by setting $\Del(\mfS^1)=\Delta(I^1_0)=\Delta(I^1_1) :=1$ and note that (\ref{eq:DeltadiamGam}) holds for $n=1$.  Now assume that for some $n\in\mfN$ and all $k\in\{0,1,\dots,2^n-1\}$, $\Delta(I^n_k)$ has been defined so that (\ref{eq:DeltadiamGam}) holds.  Consider a dyadic subarc $I^n=I^n_k$, its two  children $I^{n+1}, \tilde{I}^{n+1}\subset I^n$,  and its corresponding arc $A^n=A^n_k=\varphi(I^n_k)\subset\Gamma$.  We note that by construction each child $A^{n+1}$ of $A^n$ satisfies
\begin{gather*}
  \half\,\diam(A^n)\le\diam(A^{n+1})\le\diam(A^n) \,.  \\
  \intertext{\indent We examine two cases.  If $\Delta(I^n) \leq \diam(A^n)$, then we define}
  \Delta(I^{n+1}) =\Del(\tilde{I}^{n+1}) := \Delta(I^n)\,. \\
  \intertext{We see that \eqref{eq:DeltadiamGam} holds (for $n+1$) , since}
  \begin{align*}
      \frac{1}{2}\Delta(I^{n+1}) &= \frac{1}{2}\Delta(I^n) \leq \frac{1}{2}\diam(A^n) \leq \diam(A^{n+1})  \\
      &\leq \diam(A^n) \leq 2\Delta(I^n) = 2 \Delta(I^{n+1})\,.
  \end{align*}
\end{gather*}
Here \eqref{eq:DeltadiamGam} was used for $n$ in the last inequality.

\smallskip
If $\Delta(I^n) > \diam(A^n)$, then we define
\begin{gather*}
  \Delta(I^{n+1}) =\Del(\tilde{I}^{n+1}) := \frac{1}{2}\Delta(I^n) \,.  \\
  \intertext{Again one checks that \eqref{eq:DeltadiamGam} holds (for $n+1$), since}
  \begin{align*}
      \frac{1}{2}\Delta(I^{n+1}) &= \frac{1}{4}\Delta(I^n) \leq  \frac{1}{2}\diam(A^n) \leq \diam(A^{n+1})   \\
      &\leq \diam(A^n)  < \Delta(I^n) = 2 \Delta(I^{n+1}) \,.
  \end{align*}
\end{gather*}
Here \eqref{eq:DeltadiamGam} was used for $n$ in the first inequality.
\qed

%%-------------------------------------%%
\subsection{Proof of (B)}  \label{s:B} %%
%%-------------------------------------%%
We assume $(\Gam,\ed)$ is $1$-\bt\ with $\diam(\Gam)=1$ and doubling with finite Assouad dimension $\alf$.  Fix any $\sig\in(2^{-1/\alf},1)$ (equivalently, $\alpha< \log2/\log(1/\sigma$)).  We construct a dyadic diameter function $\Del$ on $\mcI$, using the snowflake parameter $\sig$, so that $(\Gam,\ed)$ is \bl\ equivalent to $(\mfS,d_\Del)$.  In contrast to our above proof of (A), here we do ``$m$ steps at the same time''; i.e., each arc will be divided into $2^m$ subarcs of the same diameter.  That is, we will in fact construct a $2^m$-adic diameter function;  see  \rf{s:2DF}.

\smallskip
Put $\beta:=\log2/\log(1/\sig)$, so $\sig=2^{-1/\beta}$.  Then since $\beta>\alf=\dimA(\Gam)$, there exists an $\veps_0\in(0,1]$ such that for all $\veps\in(0,\veps_0)$, the cardinality of any $\veps D$-separated set $S\subset\Gam$ with $D=\diam(S)$ satisfies
\begin{gather}
  \notag
  \card(S) < \veps^{-\beta} \,.  \\
  \intertext{Since $\sig=2^{-1/\beta}<1$, we may select an $m\in\mfN$
    so that}
  \label{eq:def_m}
  \tau:=\sig^m= \lp 2^{-1/\beta} \rp^m = \lp 2^m\rp ^{-1/\beta} < \veps_0 \,.
\end{gather}
In particular, if $S$ is a $\tau D$-separated subset of $\Gam$, with $D=\diam(S)$, then $\card(S)< \tau^{-\beta}=2^m=: M$.
% (Note that if $m=2\ell$ for some $\ell\in\mfN$, then $M=4^\ell$ and $\tau=(\sig^2)^\ell$:-)

It now follows that whenever we divide an arc $A$ of $\Gam$ into $M$ subarcs $A_k$ all with equal diameters, then
\begin{equation}  \label{E:diamG_eps}
   M^{-1} \diam(A)\leq \diam A_k \leq \tau \, \diam(A) \,.
\end{equation}
The left-hand inequality follows directly from the triangle inequality whereas the right-hand inequality holds because there are at least $M$ distinct endpoints of the subarcs $A_k$ (which are separated by $\diam A_k$) and so, by the above, these endpoints cannot be $\tau D$-separated with $D:=\diam(A)$

\smallskip
We use \rf{P:equi_diam} to divide $\Gamma$ into $M$ arcs $A^1_0, A^1_1,\dots, A^1_{M-1}$ all of equal diameter.  We iterate this procedure: assuming that arcs $A^n_k$ (with $k\in\{0,1,\dots,M^n-1\}$) have been so constructed, each arc $A^n_k$ is subdivided into $M$ subarcs $A^{n+1}_{kM +j}$ (with $
j\in\{0,1,\dots,M-1\}$) all of equal diameter; the subarcs $A_{kM+j}^{n+1}$ are labeled successively along $A_k^n$.  To avoid possible confusion, we note that all subarcs of the same arc $A^n_k$ have the same diameters, however, subarcs of different arcs $A^n_i, A^n_j$ do not necessarily have the same diameters.

\smallskip

Let $\mcJ=\bigcup_{n=0}^\infty  \mcJ^{n}$ be the family of all $M$-adic subarcs of $\mfS^1$; here $M=2^m$ and $\mcJ^{n} = \mcI^{mn}$ consists of the $M^n=2^{mn}$ subarcs of the form $J_k^n:=[k/2^{mn}, (k+1)/2^{mn}]\in \mcI^{mn}$ with $k\in\{0,1,\dots,2^{mn}-1\}$.  See the last paragraph of \rf{s:DS}.

Setting $\mcA^n:=\{A^n_k \mid k\in\{0,1,\dots,M^n-1\}\}$ (for each $n\in \mfN$) we obtain a shrinking subdivision $(\mcA^n)_1^\infty$ for $\Gam$; see \rf{s:shrinking}.  In fact, $(\mcJ^n)_1^\infty$ and $(\mcA^n)_1^\infty$  are combinatorially equivalent shrinking subdivisions, and thus by \rf{P:subdiv_homeo} there is an induced homeomorphism $\varphi\colon \mfS^1 \to \Gamma$ with $\vphi(J_k^n)=A_k^n$ for all $n\in\mfN$ and all $k\in\{0,1,\dots,M^n-1\}$.

\smallskip
Now we construct an $M$-adic diameter function $\mcJ\overset{\Del}\to(0,1]$ using the snowflake parameter $\tau$ and so that $\Del$ also satisfies the following: for all $n\in\mfN$ and for all $k\in\{0,1,\dots,M^n-1\}$,
\begin{equation}  \label{E:DeltadiamGam2}
  K^{-1} \Delta(J^n_k) \leq \diam(A^n_k) \leq K\, \Delta(J^n_k) \,,
\end{equation}
where $K:=\tau \,M$.  Once this task is completed, we can appeal to \rf{L:d BL_to_d_D_variant2k} (with $C=1$ and $2^m=M$) to assert that $\varphi:(\mfS^1,d_\Del)\to(\Gam,\ed)$ is $L$-\bl\ with $L=2\,M\,K=2\,\tau\,M^2$. %=2(4\sig)^m$.

To start, we set $\Del(\mfS^1):=1$ and then for each $k\in\{0,1,\dots,M-1\}$, we put $\Delta(J^1_k):=\tau$.  To check \eqref{E:DeltadiamGam2} for $n=1$ we use \eqref{E:diamG_eps} and the fact that $\diam(\Gam)=1$ to see that
$$
  \frac1{K} \, \Del(J_k^1)= \frac{\tau}{K} = \frac1{M} \le \diam(A_k^1) \le \tau = \Del(J_k^1) \,.
$$
Assume that for some $n\in\mfN$ and all $k\in\{0,1,\dots,M^n-1\}$, $\Delta(J^n_k)$ has been defined so that \eqref{E:DeltadiamGam2} holds.  Fix any $M$-adic subarc $J^n=J^n_k$ and let $A^n=A^n_k=\varphi(J^n_k)$ be the corresponding subarc of $\Gamma$.  We consider two cases.

First, suppose  $\Delta(J^n) \leq \diam(A^n)$.  Then we define the diameter of each child $J^{n+1}$ of $J^n$ by
$$
  \Delta(J^{n+1}):= \tau \, \Delta(J^n)\,.  \\
$$
To confirm that \eqref{E:DeltadiamGam2} is still satisfied for all these children, we observe that
\begin{align*}
  \frac{1}{K} \, \Del(J^{n+1}) &= \frac{1}{M} \, \Delta(J^n) \le \frac{1}{M} \, \diam(A^n) \le \diam(A^{n+1})   \\
                            &\le \tau \, \diam(A^n) \le \tau \,K\, \Del(J^n) = K \,\Del(J^{n+1}) \,.
\end{align*}
Here the initial inequality holds by supposition, the next two inequalities follow from \eqref{E:diamG_eps}, and the induction hypothesis gives the last inequality.

Next, suppose  $\Delta(J^n)>\diam(A^n)$.  Now we define the diameter of each child $J^{n+1}$ of $J^n$ by
$$
  \Delta(J^{n+1}):= \frac1M \, \Delta(J^n) = \frac1{2^m} \, \Del(J^n)\,.
$$
To check that \eqref{E:DeltadiamGam2} holds for all these children, we again observe that
\begin{align*}
  \frac{1}{K} \, \Del(J^{n+1}) &= \frac{1}{K\,M} \, \Delta(J^n) \le \frac{1}{M} \, \diam(A^n) \le \diam(A^{n+1})   \\
                            &\le \tau \, \diam(A^n) \le \tau \, \Del(J^n) = K \,\Del(J^{n+1}) \,.
\end{align*}
Here the initial inequality holds by the induction hypothesis, the next two inequalities follow from \eqref{E:diamG_eps}, and our supposition gives the last inequality.

This finishes the construction of an $M$-adic diameter function $\Delta$ for which \eqref{E:DeltadiamGam2} holds for all $n\in\mfN$ and all $k\in\{0,1,\dots,M^n-1\}$.

\smallskip
Having defined an appropriate $M$-adic diameter function $\Del$ on $\mcJ$, we use \rf{L:d BL_to_d_D_variant2k} to deduce that $\varphi:(\mfS^1,d_\tau)\to(\Gam,\ed)$ is $L$-\bl, where $d_\tau:=d_{\Del}$.  The $M$-adic diameter function $\Del$, constructed using the snowflake parameter $\tau$, can be extended to a dyadic diameter function $\Del$ that is constructed with the snowflake parameter $\sig=\tau^{1/m}$.  See the discussion in \rf{s:2DF}.  Let $d_\sig$ be the metric associated with the dyadic diameter function $\Del$.  According to \rf{L:DDF to 2^mDF}, the identity map $\id:(\mfS^1,d_\sig)\to(\mfS^1,d_\tau)$ is $M$-\bl.  It now follows that $(\Gam,\ed)$ is $(ML)$-\bl\ equivalent to the metric quasicircle $(\mfS^1,d_\sig)\in\mcS_\sig$.
\qed

\begin{rmk}
  We can easily adjust the previous proof to obtain a model circle constructed from a $4$-adic diameter function.  To do so, we choose $m$ in
  \eqref{eq:def_m} to be even; say, $m=2k$, so $M=4^k$.   Then we extend the $M$-adic diameter function $\mcJ\to (0,1]$ to a $4$-adic diameter function with snowflake parameter $p:=\tau^{1/k}= \sigma^2\in (4^{-1/\alpha},1)$ as described in Remark~\ref{rmk:4m4adic}. This yields a metric $d$, constructed via the $4$-adic diameter function, such that the original metric quasicircle $(\Gamma,\ed)$ is bi-Lipschitz equivalent to $(\mfS^1,d)$.  Thus the following variant of (B) holds.
\end{rmk}

\begin{cor}[(B$'$)]  \label{cor:thmB2}
  Let $(\Gamma,\abs{\cdot})$ be a metric quasicircle with finite Assouad dimension $\alpha$. Then for each $p\in (4^{-1/\alpha},1)$ there is a $4$-adic diameter function $\Delta$, constructed with snowflake parameter $p$, and an associated metric $d=d_\Delta$, such that $(\mfS^1,d)$ is \bl\ equivalent to $(\Gamma,\abs{\cdot})$.
\end{cor}

\noindent
Note that $1\leq \alpha<2$ is equivalent to $1/4 \leq 4^{-1/\alpha}<1/2$, so in this case we can choose $p\in(4^{-1/\alpha},1/2)\subset(1/4,1/2)$.

\showinfo{Note that $ML=2\tau M^3=2(8\sig)^m\le2\veps_0 8^m$.  Also, $8\sig>4$.  Finally, if we pick $m$ so that
$$
  \sig^m \le \veps_0 < \sig^{m-1} \quad\text{then}\quad \frac{\log(\veps_0)}{\log(\sig)} \le m < 1 + \frac{\log(\veps_0)}{\log(\sig)} \,.
$$
We can write $\frac{\log(\veps_0)}{\log(\sig)}$ as $\beta \log_2(1/\veps_0)$ and....then I guess the BL constant is not more than  $16\, \veps_0^{1-3\beta}$....for whatever that is worth!  \\
The real problem here is that I do not see how to get an UPPER bound on $\veps_0$.  We seem to lose quantitative control when we pick $\veps_0$.  This is related to the proof of \rf{L:lma}.}

%%--------------------------------------------%%
\subsection{Planar Quasicircles}  \label{s:C} %%
%%--------------------------------------------%%
In \ref{pf of (C)} below we corroborate part (C) of our Theorem.  Then we explain how to recover Rohde's theorem.  We begin with a precise description for the construction of Rohde snowflakes that includes some useful geometric estimates.

\smallskip
Everywhere throughout this subsection $\mcJ$ denotes the family of 4-adic subarcs of the circle $\mfS^1$.

\smallskip
Each \emph{Rohde snowflake} $R$, constructed using a parameter $p\in[1/4,1/2)$, is the Hausdorff limit of a sequence $(R^n)_1^\infty$ of polygons where $R^{n+1}$ is obtained from $R^n$ by using the replacement choices illustrated in \rf{f:Rohde_snow}.  Both the snowflake parameter $p$ and the polygonal arc $A_p$ are kept fixed throughout the construction.

We start with the unit square $R^1=E^1_0\cup E^1_1\cup E^1_2 \cup E^1_3$, so each $E^1_k$ is a Euclidean line segment of diameter one and these are labeled successively along $R^1$.  Suppose we have constructed $R^n$ as a union of $4^n$ Euclidean line segments $E^n_k$, $k\in\{0,1,\dots,4^n-1\}$ (labeled successively along $R^n$).  Then for each of the edges $E^n_k$ of $R^n$ we have two choices: either we replace $E^n_k$ with the four line segments obtained by dividing $E^n_k$ into four segments of equal diameter, or we replace $E^n_k$ by a similarity copy of the polygonal arc $A_p$ pictured at the top right of \rf{f:Rohde_snow}.  In both cases $E^n_k$ is replaced by four new line segments $E^{n+1}_{4k+j}$ (with $j\in\{0,1,2,3\}$) that we call the \emph{children} of $E^n_k$, so $E^n_k$ is the \emph{parent} of each of $E^{n+1}_{4k},E^{n+1}_{4k+1},E^{n+1}_{4k+2},E^{n+1}_{4k+3}$.  Each of these children has Euclidean diameter equal to either $(1/4)\diam(E^n_k)$ in the first case or $p\,\diam(E^n_k)$ in the second case.  The second type of replacement is done so that the ``tip'' of the replacement arc points into the exterior of $R^n$.  Then $R^{n+1}$ is the union of the $4^{n+1}$ arcs $E^{n+1}_i$ (with $i\in\{0,1,\dots,4^{n+1}-1\}$).

We call the line segments $E^n_k$ the \emph{$4$-adic edges of $R^n$}. We note that different replacement rules can be used for different edges $E^n_i$, $E^n_j$ of $R^n$.  Thus, for example, one edge could have diameter $1/4^n$ while an adjacent edge might have diameter $p^n$ (which could be much larger).  In any event, for each $n\in \mfN$ there is a natural \homeo\ $\vphi_n:\mfS^1\to R^n$ that is given by mapping each 4-adic subarc $J^n_k\subset\mfS^1$ to the 4-adic edge $E^n_k\subset R^n$.  We say that the edge $E^n_k$ \emph{corresponds to} the subarc $J^n_k$.

\smallskip
Set $\tha=\tha(p):=2\arcsin((2p)^{-1}-1)$; this is the interior angle at the ``tip'' of the arc $A_p$ in \rf{f:Rohde_snow}, but see also the left-most picture in \rf{f:Rohde2}.  Also, notice that if $A_p$ is normalized to have diameter one, then its height is $(p-1/4)^{1/2}$.

%**************%
\begin{figure}[b] %
%**************%
  %\centering
  \begin{overpic}[width=12cm,  %grid,
    tics=10]{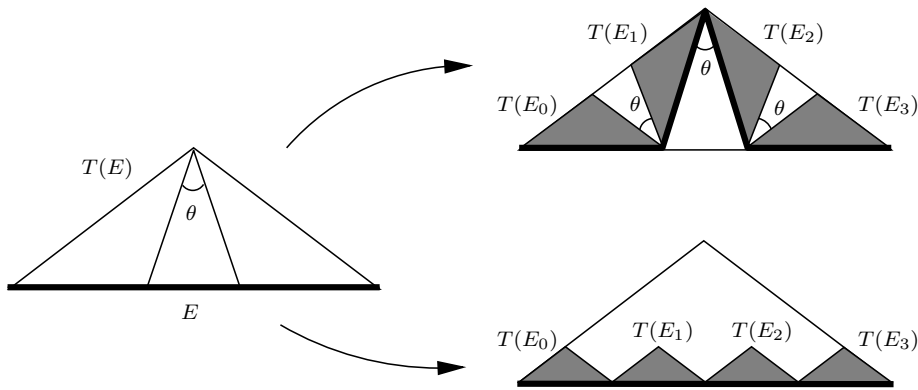}
    %
    % left figure
    \put(9,24){$\scriptstyle{T(E)}$}
    \put(20,8){$\scriptstyle{E}$}
    \put(20.5,19){$\scriptstyle{\theta}$}
    %
    % top right
    \put(55,31){$\scriptstyle{T(E_0)}$}
    \put(65,39){$\scriptstyle{T(E_1)}$}
    \put(84.5,39){$\scriptstyle{T(E_2)}$}
    \put(95,31){$\scriptstyle{T(E_3)}$}
    \put(69.7,31){$\scriptstyle{\theta}$}
    \put(77.5,35){$\scriptstyle{\theta}$}
    \put(85.7,30.6){$\scriptstyle{\theta}$}
    %
    % bottom right
    \put(55,5){$\scriptstyle{T(E_0)}$}
    \put(70,6){$\scriptstyle{T(E_1)}$}
    \put(81,6){$\scriptstyle{T(E_2)}$}
    \put(95,5){$\scriptstyle{T(E_3)}$}
  \end{overpic}
  \caption{Triangles enclosing an arc.}
  \label{f:Rohde2}
%************%
\end{figure} %
%************%

%\smallskip
Let $E$ be one of the 4-adic edges of some $R^n$.  We write $T(E)=T_p(E)$ for the closed isosceles triangle with base $E$ and height $\diam(E)(p-1/4)^{1/2}$; we orient $T(E)$ so that it ``points'' into the exterior of the polygon $R^n$.  Thus if $E$ were to be  replaced by a similarity copy of the arc $A_p$, then $T(E)$ would be the closed convex hull of this affine copy of $A_p$ (see the left-most picture in \rf{f:Rohde2}) and the third vertex of $T(E)$ would correspond to the ``tip'' of this image of $A_p$.  We call this third vertex the ``tip'' of $T(E)$.
%%%Note that the interior angle at the ``tip'' of $T(E)$ is $(\pi+\tha)/2$, and the two edges of $T(E)$ that meet at the ``tip'' both have diameters %%%equal to $\sqrt{p}\, \diam(E)$.

Next, let $E_0, E_1, E_2, E_3$ be the four children of $E$. %; so $E_j$ is some $E^{n+1}_k$ where $k\in\{4i,4i+1,4i+2,4i+3\}$.
Not only are these children contained in $T(E)$, but elementary geometric considerations reveal that the associated triangles $T(E_0),T(E_1),T(E_2),T(E_3)$ are also contained in $T(E)$.  See the two right-most pictures in \rf{f:Rohde2}.  A standard argument now reveals that the sequence $(\vphi_n)_1^\infty$ is uniformly Cauchy, and hence it converges to a continuous surjection $\vphi:\mfS^1\to R$ and the planar curve $R$ is the Hausdorff limit of the sequence $(R^n)_1^\infty$.

\showinfo{Also, the Hausdorff distance between $R^m$ and $R^n$ is at most $p^{\min\{m,n\}}$.}

Consider a subcurve $A:=\varphi(J)$ of $R$ where $J$ is some 4-adic subarc of $\mfS^1$.  Let $E$ be the 4-adic edge that corresponds to $J$.  We see that $A$ is ``built on top of $E$'' in the sense that the replacement choices used to construct $R$, applied to the edge $E$, produce $A$.  We write $A:=R(E)$ and call $A$ the \emph{$4$-adic subarc of $R$ corresponding to $E$ (and to $J$)}.  (This abuse of notation will be justified below---see \eqref{E:last eqn}---where we prove that $\vphi$ is injective, hence a \homeo, so $R$ is a Jordan curve and $A$ is an arc.)  By induction, we deduce that $A$ also lies in $T(E)$ and has the same endpoints as $E$, therefore
$$
  \diam(A)=\diam(T(E))=\diam(E)\,.
$$
%%In addition, if $\hat{E}$ is the parent of $E$,
%%%(meaning that $E$ is one of the four segments that some edge $\hat{E}\subset R^{n-1}$ is replaced by),
%%then $R(\hat{E})$ is the parent $\hat{A}$ of $A:=R(E)$.

\smallskip

Looking again at the right-most pictures in \rf{f:Rohde2}, and appealing to elementary geometric considerations, we see that the angle between any pair of consecutive triangles $T(E_0), \dots, T(E_3)$ is at least $\theta$.  It is also elementary to check that
\begin{equation}\label{E:dist(T0,T2)}\begin{split}
  \dist(T(E_0),T(E_3)) &\ge \dist(T(E_1),T(E_3)) \\ &=\dist(T(E_0),T(E_2)) \ge c(p) \, \diam(E)
\end{split}
\end{equation}
where $c(p):=\half-p$.

\smallskip

As final preparation for our proof of part (C), suppose $\hat{I},\hat{J}$ are two adjacent $4$-adic subarcs of $\mfS^1$, say with $\hat{I}\cap\hat{J}=\{\xi\}$.  (These arcs might be from different generations; i.e., possibly $\hat{I}=J^n_k$ and $\hat{J}=J^m_\ell$ where $n\ne m$.)
Let $\hat{E},\hat{F}$ be the corresponding $4$-adic edges, so $\hat{E}\cap\hat{F}=\{a\}$ where $a:= \varphi(\xi)$.

It follows from the above remarks that the angle between the two triangles $T(\hat{E})$ and $T(\hat{F})$, at their common vertex $a$, is at least $\tha$.  See \rf{f:Rohde3}.  More precisely, let $S$ be the closed sector, with vertex at $a$, that contains $T(\hat{E})$ and is such that $\tha$ is the angle between each edge of $\bd S$ and the nearest edge of $T(\hat{E})$.  Then $T(\hat{F})$ lies in the closure of $\mfR^2\sm S$.

% We pause to note two immediate consequences of this fact.  First,
% $\mfS^1\overset{\vphi}\to R$ is injective, hence a \homeo, and so
% $R\,$ is a Jordan curve.  To see that $\vphi$ is injective, fix
% distinct points $s,t\in\mfS^1$.  Let $m$ be the minimal $n\in\mfN$
% with $s\in J^n_k$, $t\in J^n_\ell$ and $J^n_k\cap I^n_\ell=\emptyset$.
% We claim that $T(E^m_k)\cap T(E^m_\ell)=\emptyset$, so
% $\vphi(s)\ne\vphi(t)$.  The claim is clear if $E^m_k$ and $E^m_\ell$
% are siblings; otherwise $\hat{E}^m_k$ and $\hat{E}^m_\ell$ are
% adjacent, but since $E^m_k\cap E^m_\ell=\emptyset$ the claim still
% holds.

% Second, each subcurve $A=R(E)$ that is ``built on top of some edge
% $E$'' is actually a subarc of $R$.  We call $A=R(E)$ a \emph{$4$-adic
%   subarc of $R$}.  There are one-to-one correspondences between the
% 4-adic subarcs of $R$, all 4-adic edges, and the 4-adic subarcs of
% $\mfS^1$: $J$ is a 4-adic subarc of $\mfS^1$ \ifff $A:=\vphi(J)$ is a
% 4-adic subarc of $R$, and then the endpoints of $A$ determine a 4-adic
% edge; equivalently, $J^n_k$ corresponds to
% $\vphi(J^n_k)=A^n_k=R(E^n_k)$.

Now suppose there is a child $E$ of $\hat{E}$ that does not contain $a$. Then $T(E)$ is compactly contained in the sector $S$ and in fact
\begin{equation}  \label{E:crucial}
  \dist(T(E), T(\hat{F})) \geq \dist(T(E), \partial S)  \geq c(p) \diam(E)
\end{equation}
where again $c(p):=\half-p$.  This follows from the estimates
$$
  \dist(T(E), \partial S) \ge \dist(b,\bd S) \geq c(p) \diam(E)
$$
where $b$ is the ``tip'' of the appropriate $T(E_0)$ as pictured in \rf{f:Rohde3}.

%**************%
\begin{figure} %
%**************%
%\centering
\begin{overpic}[width=12cm,  %grid,
    tics=10]{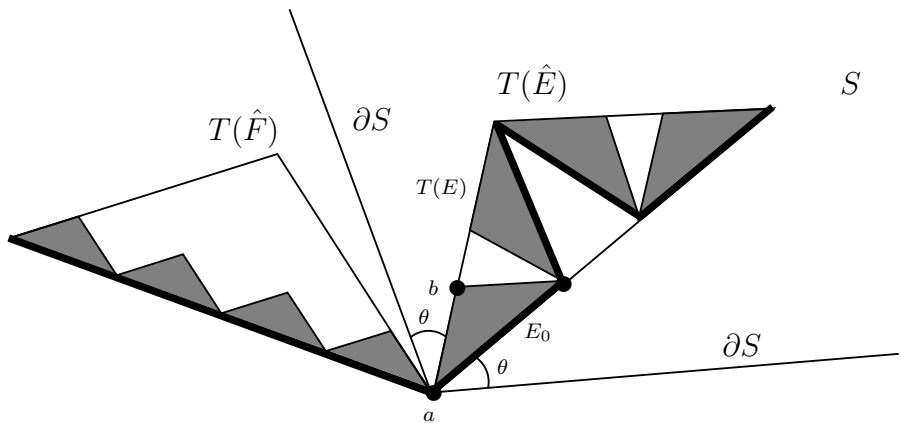}
  % label angles
    \put(46.3,8.5){$\scriptstyle{\theta}$}  \put(55,3){$\scriptstyle{\theta}$}
  % label sector and its bdry
    \put(93,34){$S$}  \put(80,5){$\bd S$}  \put(39,30){$\bd S$}
  % labels the two big triangles
    \put(55,34){${T(\hat{E})}$}
    \put(23,29){${T(\hat{F})}$}
  % label edge E and small triangle
    \put(58,7){$\scriptstyle{E_{\tiny{0}}}$}
    \put(46,23){$\scriptstyle{T({E})}$}
  % mark and label the end points
    \put(62.4,13){\circle*{1.7}}      %\put(64,12){$\scriptstyle{b}$}
    \put(50.6,12.6){\circle*{1.7}}      \put(47.4,11.8){$\scriptstyle{b}$}
    %\put(60,20){$\scriptstyle{\varphi(x)}$}
    %\put(37,7.6){$\scriptstyle{\varphi(y)}$}
  % mark and label vertex
    \put(48,0.9){\circle*{1.7}}
    \put(46.8,-2.2){$\scriptstyle{a}$}
\end{overpic}
\caption{Separating points.}
\label{f:Rohde3}
%************%
\end{figure} %
%************%

Finally, fix points $s,t\in\hat{I}\cup\hat{J}$.  Suppose there is a child $I$ of $\hat{I}$ whose interior, $\operatorname{int}(I)$, \emph{separates} $s,t$ in $\hat{I}\cup\hat{J}$ (meaning that $s,t$ lie in different components of $(\hat{I}\cup \hat{J}) \setminus \operatorname{int}(I)$).  We claim that
\begin{equation}  \label{E:last eqn}
    \abs{\vphi(s)-\vphi(t)} \geq c(p)\diam(\vphi(I)) \,.
\end{equation}
This follows from \eqref{E:dist(T0,T2)} if both $\vphi(s),\vphi(t)$ lie in $T(\hat{E})$; otherwise it follows from \eqref{E:crucial}.  Also, see \rf{f:Rohde3}.

Notice that injectivity of $\vphi$ follows from \eqref{E:last eqn}.

\smallskip
Having established the above terminology and geometric estimates, we now turn to the following.

%>>>>>>>>>>>>>>>>>>>>>>>>>>>>>>>>>>>>>>>>>>>>>>>%
\begin{pf}{Proof of {\rm(C)}}  \label{pf of (C)}%
%>>>>>>>>>>>>>>>>>>>>>>>>>>>>>>>>>>>>>>>>>>>>>>>%
We use the notation and terminology introduced above.

\smallskip
It is well-known that planar quasicircles have Assouad dimension strictly less than two; see \cite[Lemma~4.1]{Rohde-qcircles-mod-bl} or \cite[Theorem~5.2]{Luuk-ass-dim}. Furthermore, Assouad dimension is unchanged by \bl\ maps. Thus every metric quasicircle that is \bl\ equivalent to a planar quasicircle has Assouad dimension strictly less than two.

\smallskip
Let $(\Gamma,\ed)$ be a metric quasicircle with Assouad dimension $\alf\in[1,2)$.  %%; so $1/4\le (1/4)^{1/\alf}<1/2$.
We prove that $(\Gamma,\ed)$ is \bl\ equivalent to a planar quasicircle. In fact, we show that it is \bl\ equivalent to a Rohde snowflake.

Fix $p\in (4^{-1/\alpha},1/2)\subset (1/4,1/2)$.  According to part (B) of our Theorem---more precisely, the version (B$'$) stated as Corollary~\ref{cor:thmB2}---there is a $4$-adic diameter function $\Delta$ with snowflake parameter $p$ and associated metric $d_p$ such that $(\Gamma,\abs{\cdot})$ is \bl\ equivalent to $(\mfS^1,d_p)$.

% there is a metric quasicircle $(\mfS^1,d_\sig)$ in $\mcS_\sig$ that is \bl\ equivalent to $(\Gam,\ed)$, where $d_\sig$ is the metric associated with some dyadic diameter function $\Del$ that has been constructed using the snowflake parameter $\sig$.
% We assume $(\Gam,\ed)$ is 1-\bt\ with $\diam(\Gam)=1$ and
% Fix $\sig\in(2^{-1/\alf},2^{-1/2})$.

% We note that
% $$
%   p:=\sig^2 \quad\text{satisfies}\quad 1/4\le 1/4^{1/\alf}<p<1/2 \,;
% $$
% that is, $p$ is a valid snowflake parameter for the Rohde construction.  As in \rf{s:2DF}, we restrict $\Del$ to the family $\mcJ$ (of 4-adic subarcs of $\mfS^1$) thereby obtaining a 4-adic diameter function constructed using the snowflake parameter $p$.  This 4-adic diameter function has an associated metric $d_p$ and by \rf{L:DDF to 2^mDF} we know that $(\mfS^1,d_p)$ is \bl\ equivalent to $(\mfS^1,d_\sig)$, hence \bl\ equivalent to $(\Gam,\ed)$.

We use the 4-adic diameter function $\Del$ to construct a Rohde snowflake $R$ with snowflake parameter $p$, and we prove that $(\mfS^1,d_p)$ is \bl\ equivalent to $R$.  Hence $(\Gam,\ed)$ is \bl\ equivalent to a planar quasicircle.  % as asserted.

Recall that $\mcJ$ is the set of all $4$-adic subarcs of $\mfS^1$; similarly, $\mcJ^n:= \mcI^{4n}$.

\medskip

It is convenient to scale the metric $d_p$---so also the diameter function $\Del$---by the factor $1/p$.  This \bl\ change in our metric means that for each $J^1_k\in\mcJ^1$, $\Del(J^1_k)=1$.  See the paragraph immediately following \eqref{E:DeltadiamGam2}.

The desired Rohde snowflake $R$ is the limit of a sequence $(R^n)_1^\infty$ of polygons, and we must describe how to replace each edge of $R^n$ to obtain $R^{n+1}$.  Of course, we start with  the unit square $R^1:=E^1_0\cup E^1_1 \cup E^1_2 \cup E^1_3$, so each edge $E^1_k$ satisfies $\Del(J^1_k)=1=\diam(E^1_k)$. Now suppose that we have constructed polygons $R^1,R^2, \dots, R^n:=E^n_0\cup\dots\cup E^n_{4^n-1}$ so that
$$
  \text{for each }\; k\in\{0,1,\dots,4^n-1\} \,, \quad \Del(J^n_k)=\diam(E^n_k) \,.
$$

Fix any $J=J^n_k$ and consider its four children $J_0,J_1,J_2,J_3$.  Since $\Del$ is a 4-adic diameter function (constructed with the snowflake parameter $p$),
\begin{align*}
  \quad\text{either}\quad & \Del(J_0)=\Del(J_1)=\Del(J_2)=\Del(J_3):=\frac14\,\Del(J)  \\
  \quad\text{or}\quad & \Del(J_0)=\Del(J_1)=\Del(J_2)=\Del(J_3):=p\,\Del(J) \,.
\end{align*}
%%where $J_0,J_1,J_2,J_3$ are the four children of $J$.  We say that these children of $J$ are \emph{``left flat''}, or are \emph{``snowflaked''}, %%depending on whether it is the first alternative above, or the second one, that is in play.  This terminology applies to each 4-adic arc $J$.
In the first case, we replace the edge $E^n_k$ with the four segments $E^{n+1}_{4k}$, $E^{n+1}_{4k+1}$, $E^{n+1}_{4k+2}$, $E^{n+1}_{4k+3}$ obtained by dividing $E^n_k$ into four line segments of equal diameter.  Thus here $\diam(E^{n+1}_j)=(1/4)\diam(E^n_k)$.  In the second case, we replace $E^n_k$  by a similarity copy of the polygonal arc $A_p$ pictured at the top right of \rf{f:Rohde_snow}; again $E^n_k$ is replaced by four new segments $E^{n+1}_j$, but now each of these has diameter $\diam(E^{n+1}_j)=p\,\diam(E^n_k)$.  The second type of replacement is done so that the ``tip'' of the replacement arc points into the exterior of $R^n$.

It is now straightforward to check that
$$
  \text{for each }\; k\in\{0,1,\dots,4^{n+1}-1\} \,, \quad \Del(J^{n+1}_k)=\diam(E^{n+1}_k) \,.
$$
In particular, we can iterate this construction and thus obtain a sequence $(R^n)_1^\infty$ of planar polygons.  As explained above, the sequence $(R^n)_1^\infty$ converges, in the Hausdorff metric, to a Rohde snowflake $R\,$ that has been constructed using the snowflake parameter $p$.

\smallskip
Let $\mfS^1\overset{\vphi}\to R$ be the natural \homeo\ induced by the correspondences between the 4-adic subarcs of $\mfS^1$, all 4-adic edges, and the  4-adic subarcs of $R$ (see the paragraphs just before \eqref{E:dist(T0,T2)}).  Thus each 4-adic edge $E^n_k$ (of $R^n$) corresponds to a 4-adic subarc $A^n_k=R(E^n_k)=\vphi(J^n_k)$ of $R$ and
$$
  \diam(A^n_k)=\diam(T(E^n_k))=\diam(E^n_k) = \Del(J^n_k) \,.
$$
We claim that $(\mfS^1,d_p)\overset{\vphi}\to(R,\ed)$ is \bl\ with
\begin{equation}  \label{E:phi BL}
  [c(p)/8] \, d_p(s,t) \le |\vphi(s)-\vphi(t)| \le 8\, d_p(s,t) \quad\text{for all $s,t\in\mfS^1$}
\end{equation}
where $c(p):=\half-p$.

To verify this claim, let $s,t$ be two points in $\mfS^1$ and write $[s,t]$ for the smaller diameter subarc of $\mfS^1$ joining $s,t$.  Appealing to \rf{L:arc lemma_2k_var}, we get 4-adic subarcs $I,J$ of $\mfS^1$ such that: % the following hold:
\begin{align*}
  & I \cup J \subset [s,t] \subset \hat{I}\cup\hat{J} \,, \\
  & \Del(I) \le \diam_{d_p}([s,t])=d_p(s,t) \le 8\, \Del(I) \,,  \\
  & \text{$\Del(I)$ is maximal among all 4-adic subarcs in $[s,t]$} \,,  \\[-1mm]
  & \text{either $I=J$ or $\hat{I},\hat{J}$ are adjacent subarcs} \,.
\end{align*}
Here $\hat{I}, \hat{J}$ are the 4-adic parents of $I,J$.  Put $x:=\vphi(s), y:=\vphi(t)$.  Let $A:=\vphi(I),B:=\vphi(J)$ and $E,F$ be the 4-adic subarcs of $R$ and 4-adic edges (respectively) that correspond to $I,J$; also, $\hat{A}=\vphi(\hat{I}),\hat{B}=\vphi(\hat{J})$ are the parents of $A,B$.
%%
%% $\hat{A}$ and $\hat{E}$ be the parents of $A$ and $E$, so $\hat{A}=R(\hat{E})$.
%%
%%---so there are edges $E:=E^m_j$, $F:=E^\ell_k$ of certain polygons $R^m$, $R^\ell$, with $m\le\ell$, and $A:=R(E)=A^m_j$, $B:=R(F)=A^\ell_k$---
%%

Since $x,y\in \hat{A}\cup \hat{B}$,
\begin{align*}
  |x-y| &\le \diam(\hat{A}\cup\hat{B}) \le \diam(\hat{A}) + \diam(\hat{B}) = \Del(\hat{I}) + \Del(\hat{J}) \\
                &\le 4 \left[ \Del(I) + \Del(J) \right] \le 8 \,\Del(I) \le 8\,d_p(s,t)
\end{align*}
which establishes the upper estimate in \eqref{E:phi BL}.  To prove the lower estimate in \eqref{E:phi BL}, we observe that $\operatorname{int}(I)$ separates $s,t$ in $\hat{I}\cup \hat{J}$ and thus \eqref{E:last eqn} yields
$$
  |x-y| \ge c(p) \diam(\vphi(I)) = c(p) \, \Del(I) \ge  [c(p)/8] \, d_p(s,t) \,.  \qedhere
$$

\vspace*{-5mm}
%<<<<<<<%
\end{pf}%
%<<<<<<<%
%  consider
% separately the cases $I=J$ and $I\ne J$.  First, for
% $k\in\{0,1,2,3\}$, let $I_k, A_k, E_k$ denote the children of
% $\hat{I}, \hat{A}, \hat{E}$ respectively; so, these are the siblings
% of $I, A, E$ and
% $$
%   \forall \; k\in\{0,1,2,3\} \;, \quad  A_k\subset T(E_k) \quad\text{and}\quad  \diam(A_k)=\diam(E_k)=\Del(I_k) \,.
% $$

% Now suppose $I=J$.  If $s,t$ are the endpoints of $I$, then $x,y$ are the endpoints of $A$ and $E=[x,y]$, and therefore
% $$
%   d_p(s,t) = \Del(I) = \diam(E) = |x-y| \,.
% $$
% Otherwise, $s$ and $t$ belong to disjoint $I_k$'s, so $x$ and $y$ belong to disjoint $T(E_k)$'s, and then from \eqref{E:dist(T0,T2)} we see that
% $$
%   |x-y| \ge c(p) \diam(\hat{E}) = c(p)\, \Del(\hat{I}) \ge c(p)\, d_p(s,t) \,.
% $$
% %%% Notice that $c(p)^{-1}=2/(1-2p)\ge4$.

% Next, suppose $I\ne J$.  We may assume that, say, $s\in\hat{I}$ and $t\in\hat{J}$.  A careful reading of the proof of \rf{L:arc lemma_2k_var} reveals that there is at least one child of $\hat{I}$ that separates $s$ from $\hat{J}$ in $\hat{I}\cup\hat{J}$.  Thus there is an $I_k$ such that $s\in I_k$ and $I_k \cap \hat{J} = \emptyset$.  It follows that $x\in A_k\subset T(E_k)$ while $y\in\hat{B}\subset T(\hat{F})$, and also $E_k\cap\hat{F}=\emptyset$.  An appeal to \eqref{E:crucial} permits us to assert that

It is worthwhile to observe that the above provides an independent proof that each Rohde snowflake is a quasicircle; in fact, each $R$ in $\mcR_p$ is $C$-\bt\ with $C=C(p):=8/c(p)=16/(1-2p)$.

\medskip

We close this paper by explaining how Rohde's theorem follows from our
Theorem. From the proof of part (C) of our Theorem,
each planar quasicircle is \bl\ equivalent to a Rohde snowflake.
Therefore, Rohde's theorem follows from the fact that a \bl\ \homeo\
between planar quasicircles has a \bl\ extension to the entire plane.
Below we state this extension theorem, due to Gehring \cite[Theorem~7,
Corollary~2]{G-inj}, as \rf{T:qc_bL_extension}.  The construction of
the extension essentially follows from the Beurling-Ahlfors extension
\cite{BA-bdry-corr}.  See also \cite[Lemma~3]{Tukia-extension} and
\cite[Theorems~2.12,2.19]{TV-BL-ext}.

Interestingly, the property of there being such a \bl\ extension, for every \bl\ self-\homeo, is a characteristic property of quasicircles among all closed (that is, bounded, so compact) planar Jordan curves.  See \cite[Theorem~5.1]{G-extQI}.

\begin{thm}[{\cite{G-inj}}]  \label{T:qc_bL_extension}
  Each bi-Lipschitz \homeo\ %$\Gamma_1\to \Gamma_2$
  between planar quasicircles extends to a bi-Lipschitz self-homeomorphism of the plane.  The \bl\ constant for the extension depends only on the original \bl\ constant and the two original \bt\ constants. % for $\Gam_1$ and $\Gam_2$.
\end{thm}

We end by remarking that the previous theorem is false for Jordan curves. Namely a bi-Lipschitz map between planar Jordan curves $\Gamma_1,\Gamma_2$ need not have a bi-Lipschitz extension to the plane. For example let $\Gamma_1$ be a circle with two outward pointing cusps and let $\Gamma_2$ be a circle with one outward and one inward pointing cusp. It is elementary that $\Gamma_1$ and $ \Gamma_2$ are bi-Lipschitz equivalent, but any such map cannot be extended to a bi-Lipschitz map of the whole plane. This example appears already in \cite[p.388]{Rickman-curves}.

\section*{Acknowledgements}
\label{sec:Acknowledgements}

Saara Lehto and David Freeman helped the authors to understand Steffen Rohde's paper. Jussi \Va\
provided many helpful suggestions and references.

    % what else?

%%==============%%
%% BIBLIOGRAPHY %%
%%==============%%
%%---------------------------%% MiKTeX can find these, but WinEdt can_not_ :-((
\bibliographystyle{amsalpha} %% so....use flwg --with laptop-- when editing TeX
\bibliography{mrabbrev,bib}  %% file; but edit bib file in MikTeX directory!!
%\bibliography{C:/DOCUME~1/David/MyDocu~1/MyDrop~1/TeXbs/bibBS/bib}
%%---------------------------%% seems to be problems if compile with latter??

%%=============%===============%%
\end{document} %>>--THE END--<<%%
%%=============%===============%%